\documentclass[12pt]{article}
\usepackage{xparse}
\usepackage[utf8]{inputenc}
\NewDocumentCommand{\tens}{t_}
 {%
  \IfBooleanTF{#1}
   {\tensop}
   {\otimes}%
 }
\NewDocumentCommand{\tensop}{m}
 {%
  \mathbin{\mathop{\otimes}\displaylimits_{#1}}%
 }

\usepackage{amsmath,amssymb,amsthm,color,framed,mathtools}
\usepackage[hyphens]{url}
\usepackage{hyperref}
\usepackage[hyphenbreaks]{breakurl}
\usepackage[margin=1in]{geometry}
\usepackage{enumerate}
\usepackage{thmtools,thm-restate}
\usepackage{mathrsfs}
\usepackage{tikz-cd}
\tikzset{Isom/.style={above,every to/.append style={edge node={node [sloped, auto=false]{$\sim$}}}}}
\usepackage[backend=biber,style=alphabetic, citestyle=alphabetic
]{biblatex}
\usetikzlibrary{decorations.pathmorphing,arrows,arrows.meta,matrix,backgrounds}

\usepackage{xspace}
\xspaceaddexceptions{]\}}

\newcommand\simto{\stackrel{\textstyle\sim}{\smash{\longrightarrow}\rule{0pt}{0.4ex}}}
\renewcommand{\O}{\mathscr{O}}
\newcommand{\Z}{\mathbf{Z}}
\newcommand{\dublet}{{\bullet,\bullet}}

\renewcommand{\P}{\mathbf{P}}

\newcommand{\F}{\mathcal{F}}
\newcommand{\Fp}{\mathbf{F}_p}
\newcommand{\G}{\mathcal{G}}
\newcommand{\spec}{{\rm Spec}}

\newcommand{\sph}{{\rm Sph}}
\newcommand{\spf}{{\rm Spf}}
\newcommand{\Jac}{{\rm Jac}}
\newcommand{\Hom}{{\rm Hom}}

\newcommand{\colim}{\varinjlim}

\newcommand{\ol}[1]{\overline{#1}}
\newcommand{\mc}[1]{\mathcal{#1}}
\newcommand{\mf}[1]{\mathfrak{#1}}

\newcommand{\Et}{\'etale\xspace}
\newcommand{\hlfp}{hlfp\xspace}
\newcommand{\hfp}{hfp\xspace}
\newcommand{\hetale}{h-\'etale\xspace}

\newcommand{\et}{{\rm\acute{e}t}}
\newcommand{\het}{{\rm h\mbox{-}\et}}

\renewcommand{\H}{{\rm H}}
\newcommand{\cmpsn}{coherent GHGA comparison\xspace}
\newcommand{\shfcm}[1]{GHGA comparison on ${#1}$\xspace}
\newcommand{\spcite}[2]{\cite[\href{https://stacks.math.columbia.edu/tag/#1}{{#2} {#1}}]{stackp}}
\newcommand{\hfcite}[1]{\cite[#1]{hensfound}}
\newcommand{\phfcite}[1]{\cite{hensfound}, #1}
\newcommand{\propab}{Proposition}

\newcommand{\lemab}{Lemma}
\newcommand{\exab}{Example}
\newcommand{\thmab}{Theorem}
\newcommand{\corab}{Corollary}

\newcommand{\propcxyhglobsec}{\propab{} 3.1.15}
\newcommand{\exgrecowrong}{\exab{} 3.1.16}
\newcommand{\lemmodpush}{\lemab{} 3.2.1}

\newcommand{\lemhensfaithexact}{\lemab{} 3.2.3}
\newcommand{\lemmodpullback}{\lemab{} 3.2.4}
\newcommand{\lemfancypullback}{\lemab{} 3.2.5}
\newcommand{\thmthmbpos}{\thmab{} 3.2.6}

\newcommand{\lemhogen}{\lemab{} 3.2.9}
\newcommand{\prophensffulexact}{\propab{} 3.2.10}
\newcommand{\propflathenset}{\propab{} 5.2.10}
\newcommand{\propglobcatequiv}{\propab{} 5.3.3}
\newcommand{\thmqcohhet}{\thmab{} 5.3.5}

\newcommand{\lemstalkhet}{\lemab{} 5.3.7}
\newcommand{\lemthmbhet}{\lemab{} 5.3.8}
\newcommand{\thmhetzarcoh}{\thmab{} 5.3.9}
\newcommand{\corhetzarcmpsn}{\corab{} 5.3.11}

\newtheorem{theorem}{Theorem}[subsection]
\renewcommand{\thetheorem}{%
  \ifnum\value{subsection}>0
    \thesubsection
  \else
    \thesection
  \fi
  .\arabic{theorem}%
}
\newtheorem{corollary}[theorem]{Corollary}
\newtheorem{proposition}[theorem]{Proposition}
\newtheorem{lemma}[theorem]{Lemma}
\theoremstyle{definition}
\newtheorem{definition}[theorem]{Definition}
\newtheorem{remark}[theorem]{Remark}
\newtheorem{examplex}[theorem]{Example}
\newenvironment{example}
  {\pushQED{\qed}\examplex}
  {\popQED\endexamplex}

\author{Sheela Devadas\thanks{University of Washington, Seattle, WA, USA. \href{mailto:sheelad@uw.edu}{sheelad@uw.edu}}}
\title{GAGA for Henselian schemes}
\date{}
\begin{document}
\maketitle
\begin{abstract} The global analogue of a Henselian local ring is a Henselian pair---a ring $R$ and an ideal $I$ which satisfy a condition resembling Hensel's lemma regarding lifting coprime factorizations of monic polynomials over $R/I$ to factorizations over $R$. The geometric counterpart is the notion of a Henselian scheme, which can serve as a substitute for formal schemes in applications such as deformation theory. 

In this paper we prove a GAGA-style cohomology comparison result for Henselian schemes in positive characteristic, making use of a ``Henselian \Et'' topology defined in previous work in order to leverage exactness of finite pushforward for abelian sheaves in the \Et topology of schemes. We will also discuss algebraizability of coherent sheaves on the Henselization of a proper scheme, proving (without a positive characteristic restriction) algebraizability for coherent subsheaves. We can then deduce a Henselian version of Chow's theorem on algebraization and the algebraizability of maps between Henselizations of proper schemes.\\
\textbf{Keywords:} Henselian schemes, Henselian pairs, cohomology comparison, GAGA-style theorems\\
\textbf{MSC:} 14A20, 13J15
\end{abstract}\tableofcontents

\section{Introduction}\label{sec:intro}

In this paper we will discuss Henselian schemes, which are a model for algebraic tubular neighborhoods in algebraic geometry. More specifically, we will consider proper schemes over a Henselian base ring and how they compare to their Henselizations. 

As discussed in Section~\ref{chap:def}, an affine Henselian scheme $\sph(A,I)$ for a Henselian pair $(A,I)$ (a ring and ideal satisfying a condition resembling Hensel’s lemma) is a locally ringed space with underlying space $V(I) \subseteq \spec(A)$, and equipped with a richer sheaf of rings---namely, on distinguished affine opens and on stalks, it is the Henselization with respect to $I$ of the structure sheaf of $\spec(A)$. These affine Henselian schemes can be glued together to form general Henselian schemes. One example of a Henselian scheme is the Henselization of a scheme along a closed subscheme, which yields a Henselian scheme that has the same underlying space as the closed subscheme and a structure sheaf which is in some sense the Henselization of the structure sheaf of the original scheme (as defined in \cite{hensfound}). 

It is natural to ask whether similar Henselian GAGA (abbreviated as GHGA) comparison theorems to those of formal GAGA (abbreviated as GFGA) exist for Henselian schemes, such as cohomology comparison between coherent sheaves on a proper scheme over a Henselian base ring and the pullback to the Henselization of the scheme, or ``algebraizability'' for coherent sheaves on the Henselization of a scheme in the form of a categorical equivalence between coherent sheaves on a proper scheme over a Henselian base ring and coherent sheaves on the Henselization. 

\subsection{Results}\label{sub:results}

We will prove a general GHGA statement regarding proper schemes over Henselian pairs in positive characteristic.

\begin{restatable*}[Henselian cohomology comparison]{theorem}{thmposcharnonnoeth}\label{thm:poscharnonnoeth} Let $(A,I)$ be a Henselian pair such that $A$ has characteristic $p > 0$, and $X$ a proper $A$-scheme with $I$-adic Henselization $X^h$, a Henselian scheme over $\sph(A)$. 

Then for any quasi-coherent sheaf $\G$ on $X$ and any $j \ge 0$, the canonical cohomology comparison map $\H^j(X,\G) \to \H^j(X^h,\G^h)$ is an isomorphism.\end{restatable*}

The proof proceeds by a series of reductions, using the ``h-\'{e}tale" topology discussed in previous work \cite{hensfound} to isolate the role of the $p$-torsion hypothesis to the study of the cohomology of the structure sheaf on $\P^1$.

It seems to remain an open question whether every coherent sheaf on the Henselization of a proper scheme over a Henselian $\Fp$-algebra admits an algebraization. Using a characteristic 0 counterexample \cite{cohpropblog} to Theorem~\ref{thm:poscharnonnoeth} we will easily see in Example~\ref{ex:algfail} that this is not the case for the Henselized projective line in characteristic $0$.

Using Artin-Popescu approximation and formal GAGA over a complete Noetherian base, we will prove that subsheaves of algebraizable sheaves are algebraizable:

\begin{restatable*}[Algebraizability of subsheaves]{theorem}{thmsubshf}\label{thm:subshf}  Let $X \to \spec(A)$ be a proper and finitely presented morphism of schemes, with $(A,I)$ a Henselian pair. For $\F$ a finitely presented sheaf on $X$ and $\G$ a finitely presented $\O_{X^h}$-submodule of $\F^h$, there exists a finitely presented subsheaf $\G_1 \subset \F$ such that $\G_1^h \simeq \G$. \end{restatable*}

From Theorem~\ref{thm:subshf} we can deduce a Henselian version of Chow’s theorem \cite[Theorem V]{chowthm}, showing that a closed Henselian subscheme of Henselian projective space (or indeed of the Henselization of any proper, finitely presented scheme over a Henselian base) arises from a subscheme of the corresponding scheme.

\begin{restatable*}[Henselian Chow's Theorem]{corollary}{corhenschow}\label{cor:henschow}  Let $X \to \spec(A)$ be a proper and finitely presented morphism of schemes, with $(A,I)$ a Henselian pair. For a finitely presented closed Henselian subscheme $Y \subseteq X^h$, there exists a finitely presented subscheme $Z \subseteq X$ with $Z^h = Y$ as Henselian subschemes of $X^h$.  \end{restatable*}

In past work \hfcite{\prophensffulexact} it was shown that the pullback functor on coherent sheaves between a proper scheme over a Noetherian complete affine base and its Henselization is exact and fully faithful. We use Theorem~\ref{thm:subshf} and this past work to relax the completeness hypothesis and to consider the essential image of this pullback functor, proving that:

\begin{restatable*}[Henselian GAGA]{corollary}{coressimg}\label{cor:essimg}  Let $(A,I)$ be a Noetherian Henselian pair, and let $X$ be a proper $A$-scheme with $I$-adic Henselization $X^h$. Then the functor $$(\cdot)^h: \mathbf{Coh}(X) \to \mathbf{Coh}(X^h)$$ (taking a coherent sheaf $\F$ on $X$ to its pullback $\F^h$ on $X^h$) is exact and fully faithful, and its essential image is closed under subobjects and quotients. 
\end{restatable*} 

It also follows that coherent ideal sheaves, hence maps, are algebraizable in any characteristic.

\begin{restatable*}[Algebraizability of maps]{corollary}{corhensmapsalg}\label{cor:hensmapsalg}Let $X \to \spec(A), Y \to \spec(A)$ be  proper and finitely presented morphisms of schemes, with $(A,I)$ a Henselian pair. Then the map $$\Hom_{\spec(A)}(X,Y) \to \Hom_{\sph(A)}(X^h,Y^h)$$ is bijective.
\end{restatable*}

Hence for any Henselian $A$, the functor taking a proper $A$-scheme to its Henselization is fully faithful.

\subsection{Motivation and history}\label{sub:motivhist}

F. Kato has proved a Henselian cohomology comparison in degree $0$---as well as the fact that algebraizability is inherited by coherent subsheaves---in the case of a valuation base ring that is Henselian with respect to a principal ideal \cite{kato17}. We will prove a cohomology comparison for a more general base (under a $p$-torsion hypothesis) in Theorem~\ref{thm:poscharnonnoeth} and an algebraizability result for finitely presented sheaves over a general base Henselian ring without a $p$-torsion hypothesis in Theorem~\ref{thm:subshf}. It was shown by de Jong in \cite{cohpropblog} that in characteristic $0$ or mixed characteristic, cohomology comparison in degree $1$ can fail even for the structure sheaf on the projective line, so the $p$-torsion hypothesis is needed for our cohomology comparison results.

Our method of proof differs from the conventional approach to proving GAGA and GFGA theorems, since directly computing the cohomology of twists $\O(n)$ of the structure sheaf on projective space would be difficult in the Henselian setting due to lack of a concrete description for elements of the Henselization of a polynomial ring. Instead we will reduce to the study of the cohomology of the structure sheaf on $\P^1$ by leveraging a finite flat map $(\P^1)^{\times d} \to \P^d$, which we recall in Section~\ref{sec:prodlines}, as part of a series of reduction steps. However, in order to make use of this map, we must consider the effect of finite pushforward on the cohomology of a quasi-coherent sheaf on a Henselian scheme. 

For schemes, the higher direct images of a quasi-coherent sheaf by a finite morphism vanish. This is likely untrue in general for Henselian schemes, as we will discuss in more detail in Section~\ref{sub:overview}. We get around this issue using \hfcite{\corhetzarcmpsn}, which states that in positive characteristic the problem of \Et and \hetale comparison is {\it equivalent} to the comparison problem for Zariski topologies. Hence we can make our reductions of GHGA statements in the setting of \Et and \hetale cohomology and use the fact that finite pushforward on arbitrary abelian sheaves is {\it exact} for the \Et topology of schemes and hence for the \hetale topology of Henselian schemes. 

Another consequence of de Jong's counterexample \cite{cohpropblog} is that even in the case of the projective line in characteristic 0 (or in mixed characteristic) there are already failures of algebraizability for coherent sheaves, as we show in Example~\ref{ex:algfail}. However, we will be able to show in Theorem~\ref{thm:subshf} that coherent subsheaves of algebraizable coherent sheaves are algebraizable regardless of characteristic of the Henselian base ring. 

It follows from Theorem~\ref{thm:subshf} that for a proper scheme $X$ over a Noetherian Henselian affine base, the pullback functor $\mathbf{Coh}(X) \to \mathbf{Coh}(X^h)$ is exact and fully faithful, and its essential image is closed under subobjects and quotients, as we discuss in Corollary~\ref{cor:essimg}. 

Corollary~\ref{cor:essimg} is similar to a recent result in o-minimal geometry \cite[Theorem 1.4]{omin}. While o-minimal geometry concerns itself with topological spaces that are locally modeled on sets that are definable with regard to some o-minimal structure, Henselian geometry can be thought of similarly as an algebraic substitute for formal geometry. In fact the Henselization of sufficiently nice rings can be identified with the ``algebraic'' subring of their completion, as we review in Lemma~\ref{lem:hensalgcplt}.

However, while in the o-minimal setting the analogue of Theorem~\ref{thm:subshf} is deduced from the definable Chow's theorem result of Peterzil and Starchenko \cite[Corollary 4.5]{psomin}, in the Henselian case we proceed in reverse, deducing the Henselian Chow's theorem (Corollary~\ref{cor:henschow}) from Theorem~\ref{thm:subshf}.

Another similarity between Henselian geometry and o-minimal geometry is the fact that the Henselization functor $\mathbf{Coh}(X) \to \mathbf{Coh}(X^h)$ is {\it not} essentially surjective in general. We give a specific counterexample in Example~\ref{ex:algfail}. The definabilization functor $\mathbf{Coh}(X) \to \mathbf{Coh}(X^{def})$ is also not essentially surjective, as shown in \cite[Example 3.2]{omin} for $X=\mathbb{G}_m$. Although $\mathbb{G}_m$ is not proper, it is an appropriate counterexample as the o-minimal equivalent \cite[Theorem 1.4]{omin} of Theorem~\ref{thm:subshf} applies to $X$ which is a separated algebraic space of finite type.

\subsection{Dependence on previous work}\label{sub:dependence}

Many results in this paper depend on the results of \cite{hensfound}, some of which are reviewed in Section~\ref{chap:def}. (These results can also be found in my Ph.D. thesis \cite[Chapters 1-5]{sheelathesis}.) 

The GHGA cohomology comparison results of Section~\ref{chap:propcmpsn} depend on the Henselian \Et or ``\hetale'' topology defined in \cite{hensfound}. Finite pushforward is exact for the \hetale topology of Henselian schemes as a consequence of the equivalence of categories between the \hetale site of a Henselian scheme and the \Et site of the underlying scheme \hfcite{\propglobcatequiv}. Therefore (as discussed in Section~\ref{sub:motivhist}) we prove the reduction step 
Proposition~\ref{prop:finflatprojred} in the \Et/\hetale setting. This is possible because of the equivalence of Zariski and \hetale cohomology comparison in positive characteristic \hfcite{\corhetzarcmpsn}. 

In Section~\ref{chap:algzation}, Corollary~\ref{cor:essimg} and the counterexample Example~\ref{ex:algfail} depend on the results \hfcite{\lemhensfaithexact, \prophensffulexact} concerning the exactness and faithfulness of the Henselization functor $(\cdot)^h: \mathbf{Coh}(X) \to \mathbf{Coh}(X^h)$ for the Henselization $X^h$ of a scheme $X$ along a closed subscheme $Y$. Some of the lemmas on limits in the Appendix, which are used in this paper for the case of non-Noetherian rings, also depend on \cite{hensfound}. Lemma~\ref{lem:henslimcohom} depends on \hfcite{\lemmodpullback}, which describes the pullback of quasi-coherent sheaves on affine schemes to their Henselization. Lemma~\ref{lem:subschemelem} depends on the Henselian version of ``Nike's trick'' for distinguished affines \hfcite{\propab{} 4.2.1}.



\subsection{Acknowledgments}

This paper is based on Chapters 6 and 7 of my Ph.D. thesis \cite{sheelathesis}. I would like to thank my Ph.D. advisor Brian Conrad, Ravi Vakil, my postdoctoral mentor Max Lieblich, and Aise Johan de Jong for useful discussions. I would also like to thank the anonymous referee for many helpful comments, including the idea for Theorem~\ref{thm:nonnoethrelcmpsn}.

This material is based upon work supported by the National Science Foundation under Award No. 2102960. During the writing of my Ph.D. thesis I was supported by the NSF-GRFP fellowship (NSF Grant \#DGE-1147470) and by the Stanford Graduate Fellowship.

\section{Background and previous work}\label{chap:def}

In this section we review previous work from \cite{hensfound} regarding the Zariski cohomology and \hetale cohomology of a Henselian scheme which will be useful to us when proving our GHGA comparison theorems. The notion of a Henselian pair appears in \cite[Section 18.5]{ega44}, as well as the Henselization of a pair \cite[Section 18.6]{ega44}; Henselian pairs are also discussed in \cite[Chapter XI]{raynaudhens}. 
For a review of the definition of Henselian schemes and various useful statements about quasi-coherent sheaves on them, see \hfcite{Sections 2-3}.  

\begin{definition}\label{def:getususch} Given a Henselian scheme $(X,\O_X)$, the {\bf underlying scheme} of $X$ is the locally ringed space $(X,\O_X/\mc{I})$ where 
 $\mc{I}$ is the radical ideal sheaf of $\O_X$ comprising sections which vanish at the residue field of every point. 
 
We can see that $(X,\O_X/\mc{I})$ is a scheme by considering the case $X=\sph(A,I)$, for which this process gives the scheme $\spec(A/I)$.  The general case follows by gluing the affine schemes coming from each affine open. If $X$ is the Henselization of a scheme $Y$ along a reduced closed subscheme $Z$, then the preceding construction recovers $Z$. 
\end{definition}
\begin{definition}\label{def:het}
A morphism $f: (Y,\O_Y) \to (X,\O_X)$ of Henselian schemes is {\bf Henselian \'etale}, or {\bf \hetale}, if it can be locally described as a morphism of affine Henselian schemes $\sph(B,J) \to \sph(A,I)$ arising from a map of pairs $(A,I) \to (B,J)$ where $B \simeq R^h$ is the Henselization of an \Et $A$-algebra $R$, with $\sqrt{IB}=\sqrt{J}$. 

The condition of being \hetale is equivalent to being \hlfp or ``Henselian locally finitely presented''\footnote{Zariski-locally the Henselization of a finitely presented map of pairs}, flat, and satisfying an appropriate formal lifting criterion \hfcite{\propflathenset}; it follows that being \hetale can be checked using any Zariski covering.
\end{definition}

When proving our Henselian GAGA theorems, we will wish to make use of not just a Zariski cohomology comparison between a scheme and its Henselization, but also a cohomology comparison between the \Et topology of the scheme and the \hetale topology on its Henselization.

There is an equivalence of categories between the \hetale site of a Henselian scheme and \Et site of the underlying scheme:

\begin{proposition}[\cite{hensfound} \propglobcatequiv]\label{prop:globcatequiv} Let $(X,\O_X)$ be a Henselian scheme and let $(X_0,\O_{X_0})$ be the underlying scheme. The functor $\mf{F}$ from the category of Henselian schemes \hetale over $X$ to the category of schemes \Et over $X_0$, which sends $(Y,\O_Y)$ to the underlying scheme $(Y_0,\O_{Y_0})$, is an equivalence of categories.
\end{proposition}

Furthermore, quasi-coherent (Zariski) sheaves on Henselian schemes are in fact sheaves for the \hetale topology as well \hfcite{\thmqcohhet}; this fact underlies all our discussion of the \hetale site that follows. 

\begin{remark}\label{rmk:morhet} For a quasi-coherent sheaf $\F$ on a scheme $X$ with closed subscheme $Y$ (and Henselization $X^h$ along $Y$), the sheaf $(\F^h)_{\het}$ on the small \hetale site $(X^h)_{\het}$ of $X^h$ is the pullback of $\F_{\et}$ along the morphism of sites $(X^h)_{\het} \to X_{\et}$ corresponding to the morphism of locally ringed spaces $X^h \to X$. (See \hfcite{\lemfancypullback} and the discussion in \hfcite{\thmqcohhet}.) Thus to simplify notation, we will often write $(\F_{\et})^h$ for $(\F^h)_{\het}$.

It is also easily checked that pullback and ``finite pushforward'' (see \hfcite{\lemmodpush}) of quasi-coherent sheaves along morphisms of Henselian schemes is compatible with the functor $\F \mapsto \F_{\het}$. \end{remark}

It will be useful to be able to consider the stalks of sheaves on the \hetale site as in the following lemma. 

\begin{lemma}[\phfcite{\lemstalkhet}]\label{lem:stalkhet} Let $X$ be a Henselian scheme, with $X_0$ the underlying scheme of $X$. For $\ol{x}_0$ a geometric point of $X_0$ lying over $x \in X_0$, consider $x$ as a point of $X$ and write $\ol{x}$ for the geometric point of $X$ arising from $\ol{x}_0$. Then the following statements are true:

\begin{enumerate}[(i)]
\item the stalk of the structure sheaf $\O_{X_{\het}}$ of the small \hetale site at $\ol{x}$ is isomorphic to $\O_{X,x}^{\rm sh}$, the strict Henselization of the local ring $\O_{X,x}$ along its {\it maximal ideal},

\item if $X$ is the Henselization of a scheme $Y$ along a closed subscheme $Z$, then considering $\ol{x}_0$ as a geometric point of $Z=X_0 \subset Y$, the stalk of $\O_{Y_{\et}}$ at $\ol{x}_0$ is isomorphic to the stalk of $\O_{X_{\het}}$ at $\ol{x}$. 
\end{enumerate}
\end{lemma}

\begin{remark}\label{rmk:stalkhet} In the setting of Lemma~\ref{lem:stalkhet}(ii), a similar statement holds for a quasi-coherent sheaf $\F$ on $Y$. Letting $\iota: Z \to Y$ be the usual closed immersion, by \hfcite{\lemfancypullback} and the discussion in \hfcite{\thmqcohhet} we have $(\F^h)_{\het}=\iota^{-1}_\et(\F_\et)$ as sheaves on $Z_\et$. (Recall that $Z_\et$ is equivalent to $X_\het$ by Proposition~\ref{prop:globcatequiv}.) Then by the definition of the inverse image functor, the stalk of $(\F^h)_\het=(\F_\et)^h$ at $\ol{x}$ is isomorphic to the stalk of $\F_\et$ at $\ol{x}_0$ \spcite{03Q1}{Lemma}. 
\end{remark}

When computing \hetale cohomology on Henselian schemes, we will use the \hetale topology analogue of de Jong's ``Theorem B'':

\begin{theorem}[de Jong's ``Theorem B'']\label{thm:thmbpos} Let $(A,I)$ be a Henselian pair such that $A$ has characteristic $p > 0$. Then if $Z=\sph(A)$, for any quasi-coherent sheaf $\F$ on $X=\spec(A)$, the cohomologies $\H^j(Z,\F^h)$ are $0$ for $j > 0$. 
\end{theorem}
\begin{proof}
This was proved by de Jong in \cite{thmbblog}; an exposition of the proof is also given in \hfcite{\thmthmbpos}.
\end{proof}

 In \cite{nothmbblog}, de Jong provides a counterexample to the analogous vanishing for Zariski cohomology for $A$ of characteristic $0$ or mixed characteristic (see also \hfcite{\propcxyhglobsec}). The same example works in the \hetale topology, so in the following result we cannot drop the assumption that $A$ is an $\mathbf{F}_p$-algebra.

\begin{lemma}[\phfcite{\lemthmbhet}]\label{lem:thmbhet} Let $(A,I)$ be a Henselian pair with $A$ an $\Fp$-algebra. Then for a quasi-coherent sheaf $\F$ on $\spec(A)$ and any affine object $Z$ of $\sph(A)_{\het}$, the \hetale cohomologies $\H^j_{\het}(Z,\F^h)$ are $0$ for $j > 0$.
\end{lemma} 

Our cohomology comparison with the Henselization of a scheme uses the natural base change map comparing higher direct images for a map of schemes and the corresponding map of their Henselizations. For later cross-referencing purposes, we record this standard fact here:

\begin{lemma}\label{lem:basechangemap} Let $(A,I)$ be a Henselian pair and $f: X \to Y$ a morphism of $A$-schemes. Consider the commutative diagram

\begin{center}
\begin{tikzcd}
X^h \arrow[r,"p_X"] \arrow[d,"f^h"]& X \arrow[d,"f"]\\
Y^h \arrow[r,"p_Y"] &Y
\end{tikzcd}
\end{center} \vspace{-.25in}

for $X^h$ and $Y^h$ the $I$-adic Henselizations of $X$ and $Y$ respectively, and $p_X,p_Y$ the canonical maps $X^h \to X, Y^h \to Y$. 

There exists a canonical $\delta$-functorial $\O_{Y^h}$-linear morphism $(R^jf_*\F)^h \to R^jf^h_*(\F^h)$ for $\O_X$-modules $\F$.
\end{lemma} \begin{proof}  This follows from \spcite{02N7}{Lemma}, which is in the more general setting of ringed spaces with flat horizontal maps in the commutative square. \end{proof} 

%

The functor $(\cdot)^h: \mathbf{Coh}(X) \to \mathbf{Coh}(X^h)$ of pullback of coherent sheaves from a proper scheme $X$ over a base ring $A$ to the Henselization $X^h$ has been discussed in \cite{hensfound}. Exactness and faithfulness for any base is proved in \hfcite{\lemhensfaithexact}, and we can deduce fullness in the complete case from formal GAGA \cite[Theorem 5.1.4]{ega31} (see \hfcite{\lemhogen}). 


\begin{proposition}[\phfcite{\prophensffulexact}]\label{prop:hensffulexact}  Let $(A,I)$ be a Noetherian Henselian pair with $I$-adically complete $A$, and let $X$ be a proper $A$-scheme with $I$-adic Henselization $X^h$. Then the functor $(\cdot)^h: \mathbf{Coh}(X) \to \mathbf{Coh}(X^h)$ is exact and fully faithful.
\end{proposition}

In this paper, we will relax (in Corollary~\ref{cor:essimg}) the completeness hypothesis of Proposition~\ref{prop:hensffulexact}.

Just as we have a Zariski cohomology comparison map, we can compare \hetale and \Et cohomology as well. This rests on:

\begin{definition}\label{def:hetcmpsnmap}For a scheme $X$ with Henselization $X^h$ along a closed subscheme $Y$ and a quasi-coherent sheaf $\F$ on $X$, the equality $(\F^h)_{\het} = i_{\et}^{-1}(\F_{\et})$ for $i: Y \hookrightarrow X$ the canonical closed immersion gives rise to an {\bf \hetale cohomology comparison map} $\H^j(X_{\et},\F_\et) \to \H^j((X^h)_{\het},(\F^h)_{\het})$.
\end{definition}

\begin{remark}\label{rmk:ethetnotation} To simplify notation of cohomology groups, for a scheme $X$ and a quasi-coherent sheaf $\F$ on $X$ we often write $\H^j_\et(X,\F)$ for $\H^j(X_\et,\F_\et)$. Similarly for a Henselian scheme $Y$ and a quasi-coherent sheaf $\G$ on $Y$ we write $\H^j_\het(Y,G)$ for $\H^j(Y_\het,G_\het)$. 

In particular, when $Y=X^h, \G=\F^h$ we write $\H^j_\het(X^h,\F^h)$ for $\H^j((X^h)_\het, (\F^h)_\het)=\H^j((X^h)_\het,(\F_\et)^h)$ (because $(\F_\et)^h=(\F^h)_\het$; see Remark~\ref{rmk:morhet}), so the \hetale comparison map of Definition~\ref{def:hetcmpsnmap} is $\H^j_\et(X,\F) \to \H^j_\het(X^h,\F^h)$. 
\end{remark}

To conclude this background section, we include the key theorem and corollary of \cite{hensfound} in which it was proved that in positive characteristic, \hetale comparison is equivalent to Zariski comparison.

\begin{theorem}[\phfcite{\thmhetzarcoh}]\label{thm:hetzarcoh} Let $X$ be a Henselian scheme over a Henselian pair $(A,I)$ such that $A$ has characteristic $p > 0$. Then for any quasi-coherent sheaf $\G$ on $X$ and any $i \ge 0$, the natural map $\H^i(X,\G) \to \H^i_{\het}(X,\G)$ is an isomorphism.
\end{theorem}

\begin{corollary}[\phfcite{\corhetzarcmpsn}]\label{cor:hetzarcmpsn} Let $(A,I)$ be a Henselian pair such that $A$ has characteristic $p > 0$, and let $X$ be an $A$-scheme with $I$-adic Henselization $X^h$. For a quasi-coherent sheaf $\F$ on $X$, the natural comparison map $\H^j(X,\F) \to \H^j(X^h,\F^h)$ for Zariski cohomologies is an isomorphism if and only if the \hetale comparison map $\H^j_{\et}(X,\F) \to \H^j_{\het}(X^h,\F^h)$ is an isomorphism.
\end{corollary}

This equivalence of comparison problems in positive characteristic is essential to Section~\ref{chap:propcmpsn} due to non-exactness problems with finite pushforward of abelian sheaves in the Zariski topology. Since exactness of finite pushforward is necessary for some of the reduction steps in the proof of our GHGA comparison isomorphism in Theorem~\ref{thm:poscharnonnoeth}, we prove those reduction steps for the \hetale comparison problem rather than the Zariski comparison problem.

\section{Cohomology comparison on proper schemes}\label{chap:propcmpsn}

In this section and the following sections we consider comparison problems for Henselian schemes similar to formal GAGA (also called GFGA); we abbreviate ``Henselian GAGA'' as ``GHGA''. 

\begin{definition}\label{def:nhens}
A {\bf Noetherian Henselian pair} $(A,I)$ is a Henselian pair such that the ring $A$ is Noetherian.
\end{definition}

\begin{definition}\label{def:shfcm}
Let $(A,I)$ be a Henselian pair and $X$ a scheme over $\spec(A)$ with $I$-adic henselization $X^h$, a Henselian scheme over $\sph(A)$. For an $\O_X$-module $\F$, we say that $\F$ satisfies {\bf \shfcm{X}} if for any $j \ge 0$, the canonical \hetale cohomology comparison map $\H^j_{\et}(X,\F) \to \H^j_{\het}(X^h,\F^h)$ is an isomorphism.

If $A$ is Noetherian and $X$ is finite-type over $A$, we say that $X$ satisfies {\bf \cmpsn} if every coherent sheaf $\F$ on $X$ satisfies \shfcm{X}.
\end{definition}

\subsection{Overview}\label{sub:overview}

For a proper scheme $X$ over $\spec(A)$ for a Noetherian Henselian pair $(A,I)$, we wish to consider whether the comparison map $\H^i(X,\F) \to H^i(X^h,\F^h)$ is an isomorphism for coherent sheaves $\F$ on $X$. The conventional approach to proving GAGA and GFGA theorems begins by directly computing the cohomology of twists $\O(n)$ of the structure sheaf on projective space in the analytic or formal setting. However, that would be difficult to do in the setting of Henselian schemes, since it is difficult to concretely describe the elements of a Henselized polynomial ring. Furthermore, the idea of a ``graded'' $A$-module or a projective system of $A/I^n$-modules for $n \in \mathbf{N}$ does not have a natural counterpart in the Henselian setting. 

In fact it has been shown by de Jong that in characteristic 0 and mixed characteristic, \cmpsn does not hold for proper Henselian schemes \cite{cohpropblog}, even for $\F=\O_X$ on $X=\P_A^1$. Therefore we have to use alternative methods to prove \cmpsn theorems over $\mathbf{F}_p$.

Our proof of \cmpsn for proper schemes over a Noetherian Henselian base in positive characteristic proceeds as follows:

\begin{enumerate}[(I)]
\item In \hfcite{\corhetzarcmpsn} (stated here as Corollary~\ref{cor:hetzarcmpsn}) it was proven that in positive characteristic the problem of \Et and \hetale comparison is {\it equivalent} to the comparison problem for Zariski topologies. Therefore we can make the following reductions of GHGA statements in the setting of \Et and \hetale cohomology; see Remarks~\ref{rmk:whyhetale},~\ref{rmk:whyhetalechar0}.
\item We reduce \cmpsn for general projective schemes to the case of $(\P^1)^{\times d}$ by defining and making use of a finite flat map $(\P^1)^{\times d} \to \P^d$ (Section~\ref{sec:prodlines}). Specifically, we will begin our reduction steps by proving in Proposition~\ref{prop:finflatprojred} that for a finite flat surjection $P' \to P$ of projective schemes over a Henselian pair $(A,I)$, \cmpsn for $P'$ implies \cmpsn for $P$. 
\item We show in Proposition~\ref{prop:genhighdirim} that for maps satisfying a “relative comparison”, higher
direct images are compatible with pullback to the Henselization; this lets us reduce
further to the case of $\P^1$ in Proposition~\ref{prop:relp1}.
\item To reduce the proper case to the projective case and hence to the case of $\P^1$, we use Grothendieck's Unscrewing Lemma (Section~\ref{sec:propred}).
\item We complete the Noetherian case in Section~\ref{sec:cohprojline} by reducing \cmpsn for $\P^1$ to the case of the structure sheaf and then do some hands-on work; we do not need to work over $\mathbf{F}_p$ in the preceding reduction steps, but in Section~\ref{sec:cohprojline} we restrict ourselves to the setting of positive characteristic in order to return to the case of Zariski cohomology, which we can then compute. 
\item Finally in Section~\ref{sec:nonnoeth}, we will discuss the non-Noetherian case; we also extend to the case of a proper map over a base $X$ which satisfies GHGA comparison but is not necessarily affine. 

\end{enumerate}

\begin{remark}\label{rmk:whyhetale}
For schemes, the higher direct images of a quasi-coherent sheaf by a finite morphism vanish. However, this is likely {\it not true} in the setting of Henselian schemes in general.\footnote{A finite morphism of Henselian schemes refers to an \hlfp map $X \to Y$, which locally appears as a map of affine Henselian schemes $\sph(B,J) \to \sph(A,I)$ arising from a finite morphism of rings $A \to B$. The finite morphisms of Henselian schemes which we will consider are the Henselizations of finite morphisms of schemes, so we have no need for a general theory of finite morphisms in the Henselian setting.} 
The key issue is that a ring finite over a local ring which is Henselian for its maximal ideal is a finite product of local rings, but if the local base ring is Henselian along some ideal smaller than its maximal ideal, then in general we can only say that a finite algebra is semi-local; it is unknown if the higher cohomologies of quasi-coherent sheaves on the Henselian spectrum of a semi-local ring (Henselian for some ideal) vanish.

Because finite pushforward on arbitrary abelian sheaves is {\it exact} for the \Et topology of schemes and hence for the \hetale topology of Henselian schemes, the preceding issue does not arise if we make our reductions in the setting of \Et and \hetale cohomology. \end{remark}

\begin{remark}\label{rmk:whyhetalechar0} In characteristic 0 and mixed characteristic, de Jong showed that for a complete DVR $A$ the Zariski cohomology $\H^1((\P_A^1)^h,\O_{(\P_A^1)^h})$ does not vanish \cite{cohpropblog} (explained here in Example~\ref{ex:algfail}). It follows that the Zariski cohomology comparison map $\H^1(\P_A^1,\O_{\P_A^1}) \to \H^1((\P_A^1)^h,\O_{(\P_A^1)^h})$ is not an isomorphism in characteristic $0$; therefore we lose nothing by considering \hetale cohomology comparison instead. 

Although Corollary~\ref{cor:hetzarcmpsn} does not hold in characteristic 0 (its proof requires Theorem~\ref{thm:thmbpos} and Lemma~\ref{lem:thmbhet}), the argument given in \cite{cohpropblog} can be used to show nonvanishing of $\H^1_\het((\P_A^1)^h,\O_{(\P_A^1)^h})$ as well. Therefore \hetale cohomology comparison also cannot hold for $\P_A^1$, except in positive characteristic. 
\end{remark}

\subsection{Product of projective lines}\label{sec:prodlines}

The first step of reduction is to show that \cmpsn for projective $d$-space can be reduced to \cmpsn for the product of $d$ projective lines.

Fix some $d \ge 1$. For a ring $A$, we define a map of projective $A$-schemes $\varpi_A: (\P_A^1)^{\times d} \to \P_A^d$ by giving its values functorially as $$([\alpha_1,\beta_1 ],\dots,[\alpha_d,\beta_d]) \mapsto [f_0(\alpha,\beta),f_1(\alpha,\beta),\dots,f_d(\alpha,\beta)],$$ where the $f_i(\alpha,\beta)$ are homogeneous polynomials of degree $d$ in the $\alpha_j,\beta_j$ given by $$\prod_{j=1}^d (\alpha_j+\beta_jY)=\sum_{i=0}^d f_i(\alpha,\beta)Y^i.$$

\begin{proposition}\label{prop:pifinflat} The morphism $\varpi_A$ is finite, flat, and surjective.
\end{proposition}
\begin{proof} 

It suffices to consider the case $A=\Z$. It is immediately clear that $\varpi_\Z$ has finite fibers and that it is surjective (by considering its behavior on geometric points).  Furthermore, $\varpi_\Z$ is a morphism of projective $\Z$-schemes, so it is proper. Thus $\varpi_\Z$ is a finite morphism. Finally, we can see that $\varpi_\Z$ is flat by comparing dimensions of local rings in a straightforward use of ``Miracle Flatness'' applied between geometric fibers over $\spec(\Z)$ \cite[Theorem 23.1]{matscrt}.\end{proof}

For a fixed Noetherian Henselian pair $(A,I)$, we will use the map $\varpi_A$ to reduce \cmpsn for $\P_A^d$ to \cmpsn for  $(\P_A^1)^{\times d}$. More generally, we will show:

\begin{proposition}\label{prop:finflatprojred} Let $(A,I)$ be a Noetherian Henselian pair, and $\pi: P' \to P$ a finite flat surjection of projective schemes over $\spec(A)$. Then if $P'$ satisfies \cmpsn, so does $P$.
\end{proposition}

\begin{remark}\label{rmk:whynonnoeth}
In Lemmas \ref{lem:picomm}, \ref{lem:finpushhens}, and \ref{lem:pipush}, we omit the Noetherian assumption on $A$ and consider all quasi-coherent sheaves on $P'$, since Lemma~\ref{lem:pipush} will be used to reduce the non-Noetherian case for general proper schemes to the case of \cmpsn for proper schemes over a Noetherian base. For the proof of Proposition~\ref{prop:finflatprojred} we will assume that $A$ is Noetherian and consider only coherent sheaves. \end{remark}

First we show that a map $\pi$ as in Proposition~\ref{prop:finflatprojred} is compatible with pullback to the Henselization. 

\begin{lemma}\label{lem:picomm} Let $(A,I)$ be a Henselian pair, and $\pi: P' \to P$ a finite morphism of finite-type schemes over $\spec(A)$. Let  $(P')^h, P^h$ be the $I$-adic Henselizations of the  $A$-schemes $P',P$. Consider the canonical map $\pi^h$ on Henselizations such that the following diagram commutes:

\begin{centering} 

\begin{tikzcd}
(P')^h \arrow[r]\arrow[d,"\pi^h"] &P'\arrow[d,"\pi"]\\
P^h \arrow[r] & P
\end{tikzcd}

\end{centering}

Then for $\G$ a quasi-coherent sheaf on $P'$, the base change map $(\pi_*\G)^h \to \pi^h_*(\G^h)$ of Lemma~\ref{lem:basechangemap} is an isomorphism. In particular, via Remark~\ref{rmk:morhet}, the corresponding map $$((\pi_*\G)_{\et})^h=(\pi_{\et,*}\G_{\et})^h \to \pi^h_{\het,*}((\G_{\et})^h)=(\pi^h_*(\G^h))_{\het}$$ of sheaves on $(P^h)_{\het}$ is also an isomorphism.
\end{lemma} 
\begin{proof} We will show that the base change map is an isomorphism on sections over a basis of affine opens.

Since $\pi$ is finite, it is affine; therefore for each affine open $V \subset P$ with $V=\spec(R)$, the inverse image $\pi^{-1}(V)=U\subset P'$ is also affine, with $U=\spec(B)$ for a finite-type $A$-algebra $B$. Furthermore, the map $R \to B$ given by $\pi$ is a module-finite map between the finite-type $A$-algebras $R,B$.

Similarly, the preimage $V^h$ of $V$ in $P^h$ and $U^h$ of $U$ in $(P')^h$ are isomorphic to the Henselian affine schemes $\sph(R^h) \subset P^h, \sph(B^h) \subset (P')^h$ respectively (where the Henselizations are taken with respect to the ideals $IR,IB$).

We have $\G(U)=M$ for some $B$-module $M$ (since $\G$ is quasi-coherent), and $\pi_*\G|_V$ is the sheaf associated to $M$ viewed as an $R$-module via the map $R \to B$. 

We can then easily check that the base change map $(\pi_*\G)^h \to \pi^h_*(\G^h)$ of Lemma~\ref{lem:basechangemap} on $V^h$ is given by the map of $R^h$-modules $$((\pi_*\G)^h)(V^h)=M \tens_R R^h=(M \tens_B B) \tens_R R^h \to M \tens_B B^h=\G^h(U^h)=\pi^h_*(\G^h)(V^h)$$ which arises from the map $R^h \tens_R B \to B^h$; this map is an isomorphism since $R \to B$ is finite \spcite{0DYE}{Lemma}, so $((\pi_*\G)^h)(V^h) \simto \pi^h_*(\G^h)(V^h)$ via the base change map.

Therefore the base change map is an isomorphism on sections over a basis of affine opens. Hence $(\pi_*\G)^h \simto \pi^h_*(\G^h)$ as we desired to show. \end{proof}

\begin{lemma}\label{lem:finpushhens} In the situation of Lemma~\ref{lem:picomm}, for any sheaf of abelian groups $\F$ on $((P')^h)_{\het}$, the morphism $\H^i_{\het}(P^h,\pi^h_*(\F)) \to \H^i_{\het}((P')^h,\F)$ is an isomorphism. 
\end{lemma}
\begin{proof} Let $P_0,P_0'$ be the closed subschemes of $P,P'$ respectively defined by $I$. Then by Proposition~\ref{prop:globcatequiv}, we have isomorphisms of sites $(P_0)_{\et} \to (P^h)_{\het}, (P_0')_{\et} \to ((P')^h)_{\het}$. 

Let $\ol{\pi}: P_0' \to P_0$ be the base change of $\pi$. Then by \spcite{04C2}{Lemma}, the functor $\ol{\pi}_{\et,*}$ from the category of abelian sheaves on $(P_0')_{\et}$ to the category of abelian sheaves on $(P_0)_{\et}$ is exact. Hence for $j>0$ we have $R^j\ol{\pi}_{\et,*}\G=0$ for all abelian sheaves $\G$ on $(P_0')_{\et}$. 

The isomorphisms of sites above tell us that similarly, for $j > 0$ we have $R^j\pi^h_{\het,*}(\F)=0$ for all abelian sheaves $\F$ on $((P')^h)_{\het}$. This gives us the desired isomorphism of cohomology.\end{proof}

\begin{lemma}\label{lem:pipush} In the situation of Lemma~\ref{lem:picomm}, a quasi-coherent sheaf $\G$ on $P'$ satisfies \shfcm{P'} if and only if $\pi_*\G$ satisfies \shfcm{P}.
\end{lemma}
\begin{proof}
Fix some $i \ge 0$ and a quasi-coherent sheaf $\G$ on $P'$. Then we have a commutative diagram

\begin{centering}

\begin{tikzcd}[
  row sep={3em,between origins},
]
& \H^i_{\het}((P')^h,\G^h)  & \H^i_{\et}(P',\G) \arrow[l] \\
\H^i_{\het}(P^h,\pi^h_*\G^h)\arrow[Isom,ur,end anchor={[xshift=-2ex,yshift=2ex]}]& &\\
&\H^i_{\het}(P^h,(\pi_*\G)^h) \arrow[Isom,ul,start anchor={[xshift=-2ex,yshift=-2ex]}] & \H^i_{\et}(P,\pi_*\G) \arrow[l]\arrow[Isom,uu]\\
\end{tikzcd}

\vspace{-.3in}
\end{centering}

where the lower diagonal arrow is an isomorphism by Lemma~\ref{lem:picomm}. The right vertical arrow is an isomorphism because $\pi$ is finite (ensuring that $R^j\pi_{\et,*}$ vanishes on all abelian sheaves for $j > 0$). The upper diagonal arrow is an isomorphism by Lemma~\ref{lem:finpushhens}.

It is then clear that the map $\H^i_{\et}(P',\G) \to \H^i_{\het}((P')^h,\G^h)$ is an isomorphism if and only if the map $\H^i_{\et}(P,\pi_*\G) \to \H^i_{\het}(P^h,(\pi_*\G)^h)$ is an isomorphism, as we desired to show.\end{proof}

We can now prove Proposition~\ref{prop:finflatprojred}. This is a standard argument using a resolution of a coherent sheaf on $P$ by pushforwards of coherent sheaves on $P'$ and applying the hypercohomology of the resulting complex to prove our cohomology comparison. 

\begin{proof}[Proof of Proposition~\ref{prop:finflatprojred}]
Assume we have $(A,I)$ a Noetherian Henselian pair and $\pi: P' \to P$ a finite flat surjection of projective $A$-schemes. Fix a coherent sheaf $\F$ on $P$; we will show $\F$ satisfies \shfcm{P}. 

Since $\pi$ is flat and surjective---hence faithfully flat---the map $\F \to \pi_*\pi^*\F$ is injective for any coherent sheaf $\F$ on $P$. 
 
 Setting $\G_0:=\pi^*\F$, which is a coherent sheaf on $P'$, we can replace $\F$ with $(\pi_*\G_0)/\F$ to obtain an injective map $(\pi_*\G_0)/\F \to \pi_* \G_1$ for some coherent sheaf $\G_1$ on $P'$. We can then iteratively construct a resolution
$0 \to \F \to \pi_*\G_0 \to \pi_*\G_1 \to \dots$ where the $\G_i$ are coherent sheaves on $P'$.

In other words, for any coherent sheaf $\F$ on $P$, there exists a resolution of $\F$ by pushforwards of coherent sheaves on $P'$. These sheaves $\pi_*\G_i$ satisfy \shfcm{P} by Lemma~\ref{lem:pipush}. 
 
 Now we consider two complexes of coherent $\O_P$-modules: $\F[0]^\bullet$, the complex with the only nonzero term being $\F$ in degree $0$, and the complex $\G^\bullet$ with degree $i$ term $\G^i=\pi_*\G_i$. The injection $\F \hookrightarrow \pi_*\G_0$ in degree $0$ and zero maps $0 \to \pi_*\G_i$ in degree $i$ for $i>0$ give a quasi-isomorphism of complexes $\F[0]^\bullet \to \G^\bullet$. The corresponding complexes of $\O_{P_{\et}}$-modules $(\F[0]^\bullet)_{\et}$ and $(\G^\bullet)_{\et}$ are also quasi-isomorphic, so we may use the hypercohomology of $(\G^\bullet)_{\et}$ to compute the \Et cohomology of $\F$. By a similar argument, for $\mc{H}^\bullet$ the complex with degree $i$ term $\mc{H}^i=(\pi_*\G_i)^h$, the hypercohomology of $(\mc{H}^\bullet)_\het$ may be used to compute the \hetale cohomology of $\F^h$. 
 
The remainder of the proof is a standard application of the hypercohomology spectral sequence. \end{proof}

\subsection{Relative $\P^1$}\label{sec:relp1}

In the previous section, we reduced \cmpsn for $\P^d$ over a Henselian pair to \cmpsn for $(\P^1)^{\times d}$ over a Henselian pair. In this section we reduce to the case of $\P_A^1$ for $(A,I)$ a Henselian pair. More specifically, we will show:

\begin{proposition}\label{prop:relp1} Let $(A,I)$ be a Noetherian Henselian pair. Assume that for any $A$-algebra $B$ of essentially finite-type with $I$-adic Henselization $B^h$, the $B^h$-scheme $\P_{B^h}^1$ satisfies \cmpsn.  Then if a projective $A$-scheme $X$ satisfies \cmpsn, so does $Y:=\P_X^1$. 
\end{proposition}

\begin{remark}\label{rmk:relp1hypo} We assume $X$ projective in Proposition~\ref{prop:relp1} since it will be applied when $X$ is a power of the projective line. However the proof of Proposition~\ref{prop:relp1} only uses the fact that $X$ is finite-type over a Noetherian ring (hence locally Noetherian). In Theorem~\ref{thm:nonnoethrelcmpsn} we generalize Proposition~\ref{prop:relp1} to the case where $Y$ is proper over $X$ and $X$ is finitely presented over the base Henselian pair. 
\end{remark}

We will first show that higher direct images along a relative $\P^1$ map are compatible with pullback to the Henselization. In fact we can show this is true more generally for maps satisfying a ``relative comparison''. 

\begin{proposition}\label{prop:genhighdirim} Let $(A,I)$ be a Henselian pair with $A$ Noetherian, and $S$ a proper $A$-scheme. Assume that for any  $A$-algebra $B$ of essentially finite-type with $I$-adic Henselization $B^h$, the proper $B^h$-scheme $S_{B^h}=S \times_{\spec(A)} \spec(B^h)$ satisfies \cmpsn. 

For $X$ a finite-type $A$-scheme, let $Y=X \times_{\spec(A)} S$ and let $f:Y \to X$ be the natural proper map. For $\F$ a coherent sheaf on $Y$ and any $j \ge 0$, the map $(R^jf_{\et,*}\F_{\et})^h \to R^jf^h_{\het,*}(\F^h)$ of sheaves on $(X^h)_{\het}$ arising from the base change map of Lemma~\ref{lem:basechangemap} is an isomorphism.
\end{proposition} 

\begin{proof} We will show the base change map is an isomorphism by checking on stalks, using the results of \cite[Section 5.9]{leifu} on passage to limits for \Et cohomology. 

Since $f$ is proper and $X$ is locally Noetherian, the higher direct image sheaves $R^jf_{\et,*}\F_\et$ are coherent on $X$.
Therefore the pullback $(R^jf_{\et,*}\F_{\et})^h$ to $(X^h)_\het$ is a coherent sheaf of $\O_{(X^h)_{\het}}$-modules. 

We fix a point $x \in X^h$, and also write $x$ for its image in $X$. Choose a geometric point $\ol{x}$ of $X$ lying over $x$, and write $\ol{x}^h$ for the corresponding geometric point of $X^h$ lying over $x$. 
By Lemma~\ref{lem:stalkhet}, the stalks of $\O_{(X^h)_{\het}}$ at $\ol{x}^h$ and of $\O_{X_{\et}}$ at $\ol{x}$ are both naturally isomorphic to the strict Henselization $(\O_{X,x})^{\rm sh}$ of the local ring $\O_{X,x}$ along its maximal ideal. 

By the construction of higher direct images as a sheafification, we can compute the stalks of both sheaves at $\ol{x}^h$ as:\vspace{-.25in}

\begin{align}
\label{eqn:stalk1}\tag{*}
((R^jf_{\et,*}\F_{\et})^h)_{\ol{x}^h} &=  \colim_{(U,\ol{u})} \H^j_{\et}(U \times_X Y,\F),\\
\label{eqn:stalk2}\tag{**}
(R^jf^h_{\het,*}(\F^h)_{\het})_{\ol{x}^h} &= \colim_{(U,\ol{u})} \H^j_{\het}(U^h \times_{X^h} Y^h,\F^h)
\end{align} 

with $(U,\ol{u}) \to (X,\ol{x})$ the affine \Et neighborhoods. This holds because the system of \hetale neighborhoods $(U^h,\ol{u}^h: U^h \to X^h)$ of $\ol{x}$ (Henselizations of affine \Et neighborhoods of $\ol{x}$ as a geometric point of $X$) is cofinal, since by the isomorphism of sites $(X_0)_{\et} \to (X^h)_{\het}$ for $X_0$ the closed subscheme of $X$ defined by the ideal $I$ (Proposition~\ref{prop:globcatequiv}), we can check cofinality modulo $I$. The limit of the affine \Et neighborhoods $(U,\ol{u})$ of $(X,\ol{x})$ is the map $\spec((\O_{X,x})^{\rm sh}) \to X$, which we denote $X_{(\ol{x})} \to X$. 

 Our map of sheaves $(R^jf_{\et,*}\F_{\et})^h \to R^jf^h_{\het,*}(\F_{\et}^h)$ is given on each stalk as the limit of the \hetale comparison maps $\H^j_{\et}(U \times_X Y,\F) \to \H^j_{\het}((U \times_X Y)^h,\F^h)=\H^j_{\het}(U^h \times_{X^h} Y^h,\F^h)$ for each \Et neighborhood $U$ of $\ol{x}$ as described above. 
 
We begin by computing the first stalk $((R^jf_{\et,*}\F_{\et})^h)_{\ol{x}^h}$. Again by the construction of higher direct images as a sheafification, we note that that the colimit in (\ref{eqn:stalk1}) is also equal to $(R^jf_{\et,*}\F_\et)_{\ol{x}}$, which is isomorphic to $\H^j_{\et}(Y \times X_{(\ol{x})},\F|_{Y \times X_{(\ol{x})}})$ \cite[Corollary 5.9.5]{leifu}. 

%
%

We now consider the other stalk, $(R^jf^h_{\het,*}(\F_{\et}^h))_{\ol{x}^h}$. For each affine \Et neighborhood $(U,\ol{u})$ of $(X,\ol{x})$, its Henselization is the Henselian affine \hetale neighborhood $(U^h,\ol{u}^h: U^h \to X^h)$ of $(X^h,\ol{x}^h)$. If $U=\spec(B)$, then $U^h$ is isomorphic to $\sph(B^h)$ for $B^h$ the $I$-adic Henselization of $B$.  The limit $\colim B^h$ is equal to $\colim B$, which is $(\O_{X,x})^{\rm sh}$ (see Lemma~\ref{lem:stalkhet}), so the limit of the Henselian affine \hetale neighborhoods $(U^h,\ol{u}^h: U^h \to X^h)$ is $X_{(\ol{x})} \to X$. By the definition of $f$, we see that $(U \times_X Y)_{B^h} \simeq S_{B^h}$ as $B^h$-schemes. Furthermore, we see that $U^h \times_{X^h} Y^h \simeq (S_{B^h})^h$, the $I$-adic Henselization of the scheme $S_{B^h}$. The $A$-algebra $B$ is finite-type since $X$ is finite-type over $A$. Therefore by our assumption, $S_{B^h}$ \emph{satisfies \cmpsn}, so by Corollary~\ref{cor:hetzarcmpsn} the comparison map $\H^j_{\et}((U \times_X Y)_{B^h},\F_{B^h}) \to \H^j_{\het}(U^h \times_{X^h} Y^h,\F^h)$ is an isomorphism.

Since $\colim B^h = \colim B = (\O_{X,x})^{\rm sh}$, we note that both the limit of the maps $U \times_X Y \to Y$ and the limit of the maps $(U \times_X Y)_{B^h} \to Y$ are equal to the map $Y \times X_{(\ol{x})} \to Y$. Then from (\ref{eqn:stalk2}), we can compute \vspace{-.25in}

\begin{align*}
(R^jf^h_{\het,*}(\F^h)_{\het})_{\ol{x}^h} &= \colim_{(U,\ol{u})} \H^j_{\het}(U^h \times_{X^h} Y^h,\F^h)\\
&=\colim_{(U=\spec(B),\ol{u})}\H^j_{\et}((U \times_X Y)_{B^h},\F_{B^h}) \\
&= \H^j_{\et}(Y \times X_{(\ol{x})},\F|_{Y \times X_{(\ol{x})}}),
\end{align*}

with the final equality as a consequence of \cite[Corollary 5.9.4]{leifu}. 

This identification arises from the natural maps, so the maps on stalks $((R^jf_{\et,*}\F_{\et})^h)_{\ol{x}^h} \to (R^jf^h_{\het,*}(\F_{\et}^h))_{\ol{x}^h}$ are isomorphisms; hence the base change map is an isomorphism as we desired to show.\end{proof}

\begin{corollary}\label{cor:henshighdirim} In the situation of Proposition~\ref{prop:relp1}, let $f$ be the canonical map $Y=\P_X^1 \to X$, with $f^h: Y^h \to X^h$ the map induced by $f$ on the $I$-adic Henselizations. For $\F$ a coherent sheaf on $Y$ and any $j \ge 0$, the map $(R^jf_{\et,*}\F_{\et})^h \to R^jf^h_{\het,*}(\F^h)$ of sheaves on $(X^h)_{\het}$ arising from the base change map of Lemma~\ref{lem:basechangemap} is an isomorphism.
\end{corollary}
\begin{proof} This is a straightforward application of Proposition~\ref{prop:genhighdirim} with $X$ projective and $S=\P_A^1$. \end{proof}

We can now prove Proposition~\ref{prop:relp1}. 

\begin{proof}[Proof of Proposition~\ref{prop:relp1}] Let $f: Y = \P_X^1 \to X$ be the natural map. Fix a coherent sheaf $\F$ on $Y$. We will compare the Leray spectral sequence  $\H^i_{\et}(X,R^jf_*\F) \implies \H^{i+j}_{\et}(Y,\F)$  to the analogous spectral sequence for $f^h$ and $\F^h$.  

We choose an injective resolution $\mc{I}^\bullet$ of $\F_{\et}$ by $\O_{Y_{\et}}$-modules, with which we can compute the \Et cohomology of $\F$. Similarly, we choose an injective resolution $\mc{J}^\bullet$ of $\F^h_{\het}$ by $\O_{(Y^h)_{\het}}$-modules to compute the \hetale cohomology of $\F^h$. 

We may pull back $\mc{I}^\bullet$ to an exact sequence of $\O_{(Y^h)_{\het}}$-modules (with degree $i$ term $(\mc{I}^i)^h$), which maps to $\mc{J}^\bullet$ since $\mc{J}^\bullet$ is an injective resolution. This gives us a map from the pullback of $\mc{I}^\bullet$ to $(Y^h)_{\het}$ to $\mc{J}^\bullet$, which can be lifted through the steps of constructing the Leray spectral sequences (applying $f_{\et,*}$ or $f^h_{\het,*}$, taking a Cartan-Eilenberg resolution, taking global sections, and finally filtering the total complex by rows) to give us a map of spectral sequences that must respect the filtration on limit terms.

Since $f$ is projective (hence proper) proper and $X$ is locally Noetherian, the higher direct image sheaves $R^jf_{\et,*}\F_\et$ are coherent on $X$.
Therefore the pullback $(R^jf_{\et,*}\F_{\et})^h$ to $(X^h)_\het$ is a coherent sheaf of $\O_{(X^h)_{\het}}$-modules. Then by Proposition~\ref{prop:genhighdirim}, we see that pullback to the Henselization gives an isomorphism on the second sheet, since $$\H^i_{\et}(X,R^jf_*\F) \to \H^i_{\het}(X^h,(R^jf_*\F)^h) \simeq \H^i_{\het}(X^h,R^jf^h_*\F^h)$$ is an isomorphism by \cmpsn for $X$.

Since the map of spectral sequences is an isomorphism at the second sheet and respects the filtration on limit terms, the comparison map $\H^n_{\et}(Y,\F) \to \H^n_{\het}(Y^h,\F^h)$ is an isomorphism for all $n$.\end{proof}

\subsection{Reduction of the general proper case}\label{sec:propred}

We next consider general projective schemes.

\begin{lemma}\label{lem:genprojred} 
Let $(A,I)$ be a Noetherian Henselian pair. Assume that for any  $A$-algebra $B$ of essentially finite-type with $I$-adic Henselization $B^h$, the $B^h$-scheme $\P_{B^h}^1$ satisfies \cmpsn.  Then any projective $A$-scheme $X$ satisfies \cmpsn.\end{lemma}
\begin{proof} By Proposition~\ref{prop:relp1} iterated and our assumption, we see that for any $d \ge 1$ the projective $A$-scheme $(\P_A^1)^{\times d}$ satisfies \cmpsn. Then using the map $\varpi_A: (\P_A^1)^{\times d} \to \P_A^d$ described in Section~\ref{sec:prodlines}, by Proposition~\ref{prop:finflatprojred}, we see that $\P_A^d$ satisfies \cmpsn for all $d \ge 1$.

Now consider an arbitrary projective $A$-scheme $X$. Then there exists $N$ and a closed immersion $\iota: X \hookrightarrow \P_A^N$. We apply Lemma~\ref{lem:pipush} to $\iota: X \hookrightarrow \P_A^N$ to see that a coherent sheaf $\F$ on $X$ satisfies \shfcm{X} if and only if $\iota_*\F$ satisfies \shfcm{\P_A^N}.

Therefore since $\P_A^N$ satisfies \cmpsn, $X$ does as well.\end{proof}

In order to extend Lemma~\ref{lem:genprojred} to the proper case via Chow's lemma, we will use the following result:

\begin{lemma}\label{lem:henshighdirim}
Let $(A,I)$ be a Noetherian Henselian pair. Assume that for any  $A$-algebra $B$ of essentially finite-type with $I$-adic Henselization $B^h$, the $B^h$-scheme $\P_{B^h}^1$ satisfies \cmpsn.

Let $X$ be a finite-type $A$-scheme and $f: Y \to X$ be a morphism of schemes which is projective locally on $X$. For $\F$ a coherent sheaf on $Y$ and any $j \ge 0$, the map $(R^jf_{\et,*}\F_{\et})^h \to R^jf^h_{\het,*}(\F^h)$ of sheaves on $(X^h)_{\het}$ arising from the base change map of Lemma~\ref{lem:basechangemap} is an isomorphism.\end{lemma} 

\begin{proof} This can be checked affine-locally on $X$, so we can assume $X$ is affine and then that $Y=\P_X^n$. Then we can apply Proposition~\ref{prop:genhighdirim} with $S=\P_A^n$ to get the desired result by Lemma~\ref{lem:genprojred}.\end{proof}

We can now proceed with reducing the proper case to the projective case.

\begin{theorem}\label{thm:propred} Let $(A,I)$ be a Noetherian Henselian pair. Assume that for any $A$-algebra $B$ of essentially finite-type with $I$-adic Henselization $B^h$, the $B^h$-scheme $\P_{B^h}^1$ satisfies \cmpsn. Then any proper $A$-scheme $X$ satisfies \cmpsn.
\end{theorem}
\begin{proof} We will use \cite[Theorem 3.1.2]{ega31}, sometimes called ``Grothendieck's Unscrewing Lemma'' or ``Grothendieck's d\'{e}vissage theorem''.

 If we consider the full subcategory $\mc{C}$ of coherent $\O_X$-modules consisting of $\F$ satisfying \shfcm{X}, then it is obvious that $0 \in \mc{C}$ and that for any short exact sequence of coherent $\O_X$-modules $\F' \hookrightarrow \F \twoheadrightarrow \F''$, if two of $\F,\F',\F''$ are in $\mathcal{C}$, all three are. (The latter statement follows by the $\delta$-functoriality of the \hetale comparison maps and the five lemma.) 
 
By the Unscrewing Lemma, in order to show that {\it every} coherent $\O_X$-module $\F$ satisfies \shfcm{X} (or equivalently, that $\mathcal{C}$ contains all coherent $\O_X$-modules), it is enough to show that for any irreducible closed subset $Z \subset X$, there exists some $\G$ supported on $Z$ and satisfying \shfcm{X} whose fiber at the generic point of $Z$ has rank $1$.

We now use Noetherian induction on $X$. For our inductive assumption, we assume that for any strict closed subscheme $Y$ of $X$ (strict meaning $Y \subsetneq X$) that every coherent $\O_Y$-module satisfies \shfcm{Y}. Note that $Y$ is necessarily proper over $A$. If $\G$ is a coherent $\O_X$-module which is supported on a strict closed subscheme $Y$ of $X$, then considering $\G$ as a coherent $\O_Y$-module, it follows from the inductive assumption that $\G$ satisfies \shfcm{X}.

If $X$ is reducible or not reduced, then any irreducible closed subset $Z \subset X$ must be contained in a strict closed subscheme $Y$ of $X$, so any $\G$ supported on $Z$ whose fiber at the generic point of $Z$ has rank $1$ must satisfy \shfcm{X} by the inductive assumption. Thus by the Unscrewing Lemma, every coherent $\O_X$-module $\F$ satisfies \shfcm{X} if $X$ is reducible or not reduced. 

It is therefore enough to consider the case where $X$ is integral and, given the inductive assumption, exhibit a coherent $\O_X$-module $\F$ with generic stalk of rank $1$ satisfying \shfcm{X}. Then by the Unscrewing Lemma we are done.

Since $X$ is integral, we can use Chow's Lemma \spcite{02O2}{Section} to find an integral projective $A$-scheme $X'$ and a morphism $\pi: X' \to X$ over $A$ such that $\pi$ is projective and surjective, as well as a dense open $U \subset X$ such that $\pi|_{\pi^{-1}(U)}: \pi^{-1}(U) \to U$ is an isomorphism. 

Now let $\G=\O_{X'}(n)$ with $n>0$. Since $\pi|_{\pi^{-1}(U)}$ is an isomorphism, we see that the generic stalk of the coherent $\O_X$-module $\pi_*\G$ is invertible. Therefore it will suffice to show that $\pi_*\G$ satisfies \shfcm{X} for some $n$.

Consider the Leray spectral sequences for $\G_\et$ and $(\G_\et)^h$. Thus we get a spectral sequence of $A$-modules $E_\bullet^\dublet$ with second sheet $E_2^{i,j} = \H^i_{\et}(X,R^j\pi_{*}\G)$ which converges to $\H^{i+j}_{\et}(X',\G)$ (inducing finite filtrations on the limit terms), and another spectral sequence of $A$-modules  $Q_\bullet^\dublet$ with second sheet $Q_2^{i,j} = \H^i_{\het}(X^h,R^j\pi^h_*(\G^h))$ which converges to $\H^{i+j}_{\het}((X')^h,\G^h)$ (inducing finite filtrations on the limit terms).

Working through the construction of these Leray spectral sequences, we have a morphism $\phi:E_\bullet^\dublet \to Q_\bullet^\dublet$ coming from the pullback of sheaves on $X'$ to sheaves on $(X')^h$ (see also discussion in the proof of Proposition~\ref{prop:relp1}). Since for all $j$ we have $((R^j\pi_*\G)_{\et})^h=((R^j\pi_*\G)^h)_{\het}$ by Remark~\ref{rmk:morhet} and the map $((R^j\pi_*\G)^h)_{\het} \to (R^j\pi^h_*(\G^h))_{\het}$ is an isomorphism by Lemma~\ref{lem:henshighdirim}, we see that $\phi:E_\bullet^\dublet \to Q_\bullet^\dublet$ is compatible with pullback of sheaves on $X$ to sheaves on $X^h$.

 Furthermore, since we have \cmpsn for $X'$ by Lemma~\ref{lem:genprojred}, we know that this morphism is an isomorphism on the $\infty$-sheet (because the natural map between the limit terms is an isomorphism). 

The higher direct image sheaf $(R^j\pi_*\G)_{\et}=R^j\pi_{\et,*}\G_{\et}$ is the sheafification of the presheaf \linebreak$V \mapsto \H^j_{\et}(\pi^{-1}(V),\G)$. For $n$ sufficiently large, independent of the affine open $V$, and $j > 0$, each of the cohomology groups $\H^j_{\et}(\pi^{-1}(V),\G)=\H^j_{\et}(\pi^{-1}(V),\O_{X'}(n))$ vanishes. Therefore for $j > 0$, the sheaf $(R^j\pi_*\G)_{\et}$ is $0$. Hence we also have $0=((R^j\pi_*\G)_{\et})^h = ((R^j\pi_*\G)^h)_{\het} \simto (R^j\pi^h_*(\G^h))_{\het}$ for $j > 0$ (see Remark~\ref{rmk:morhet}, Lemma~\ref{lem:henshighdirim}). 

Therefore both spectral sequences $E_\bullet^\dublet$ and $Q_\bullet^\dublet$ degenerate at the second sheet. Then since our morphism of spectral sequences induces an isomorphism on the $\infty$ sheet, the same is true for the second sheet, meaning the comparison map $\H^i_{\et}(X,\pi_*\G) \to \H^i_{\het}(X^h,\pi^h_*(\G^h)) \simeq \H^i_{\het}(X^h,(\pi_*\G)^h)$ is an isomorphism for all $i$. Thus $\pi_*\G$ satisfies \shfcm{X} for large $n$. \end{proof}

\subsection{Cohomology of the Henselian projective line}\label{sec:cohprojline}

In the previous sections we reduced the problem of \cmpsn for a proper scheme over a Noetherian Henselian pair $(A,I)$ to the problem of \cmpsn for $\P_{B^h}^1$ with $(B^h,IB^h)$ the $I$-adic Henselization of a general essentially finite-type $A$-algebra $B$.

For our more explicit computations with the projective line, we will need to assume that $A$ has characteristic $p > 0$ in order to leverage the vanishing of higher cohomologies for quasi-coherent sheaves on affine Henselian schemes (as holds over $\Fp$). The Zariski version of this ``Theorem B'' for Henselian schemes is stated here in Theorem~\ref{thm:thmbpos}, and the \hetale version is stated here in Lemma~\ref{lem:thmbhet}.

 Now we consider the Henselian projective line and show that for GHGA over $(A,I)$ (for $A$ over $\Fp$) it suffices to consider just $\H^1_{\et}(\P_A^1,\O_{\P_A^1})$.
 
\begin{proposition}\label{prop:p1red} Let $(A,I)$  be a Henselian pair with $A$ Noetherian and such that $A$ has characteristic $p > 0$.  Let $P=\P_A^1$ be the projective line over $A$, with $P^h$ its $I$-adic Henselization. Assume that $\O_P$ satisfies \shfcm{P}. Then all coherent sheaves $\F$ on $P$ also satisfy \shfcm{P}. \end{proposition}
\begin{proof} 

We start by considering twists of the structure sheaf $\O_P(n)$ for all integers $n$. 

Write $P = {\rm Proj}(A[x_0,x_1])$. Then for each integer $n$ we have a short exact sequence of sheaves on $P$: $$0 \to \O_P(n-1) \xrightarrow{x_1} \O_P(n) \to \G_n \to 0$$ where the first map is given by multiplication by $x_1$. Then $\G_n$ is the structure sheaf of the affine section $[1,0] \in \P_A^1=P$, so $\G_n$ satisfies \shfcm{P} by Lemma~\ref{lem:thmbhet}.

Since pullback to the Henselization is exact (as the maps of local rings are flat), we have a short exact sequence of sheaves on $P^h$ $$0 \to (\O_P(n-1))^h \xrightarrow{x_1} (\O_P(n))^h \to \G_n^h \to 0.$$ These short exact sequences give us maps of the corresponding long exact sequences of \Et cohomology and of \hetale cohomology.  If $\G_n$ and $\O_P(n)$ satisfy \shfcm{P}, then the same is true for $\O_P(n-1)$. Thus using downwards inductions starting with $n=0$, we see that for all $n \le 0$ the twist $\O_P(n)$ satisfies \shfcm{P}.

Similarly, we see that if $\G_n$ and $\O_P(n-1)$ satisfy \shfcm{P}, then the same is true for $\O_P(n)$. Now using induction upwards from $n=1$, we see that for all $n > 0$ ---hence for any integer $n$---the twist $\O_P(n)$ satisfies \shfcm{P}.

For a general coherent sheaf $\F$ on $P$, we have a short exact sequence of sheaves $$0 \to \mc{K} \to \bigoplus_{\ell=0}^r \O_P(n_\ell) \to \F \to 0$$ for some finite collection of integers $n_0,\dots,n_r$. Note that $\mc{K}$ is also coherent as the kernel of a map between coherent sheaves. 

For $j \ge 2$, we know that $\H^j_{\et}(P,\G)=0$ for all quasi-coherent $\G$. Similarly by Lemma~\ref{lem:thmbhet}, $\H^j(P^h,\G^h)=0$ for $j  \ge 2$, because we can compute the \hetale cohomology of $\G^h$ on $P^h$ with the \v{C}ech complex associated to the Henselization of the standard affine open cover of $P$ by \spcite{03F7}{Lemma}. (Note that the Henselization of the standard affine open cover of $P$ is necessarily also an \hetale cover.)

Therefore the comparison map $\H^j_{\et}(P,\G) \to \H^j_{\het}(P^h,\G^h)$ is an isomorphism for $j \ge 2$ and all coherent $\G$. Now we use downwards induction on $j$, assuming that for all coherent sheaves $\G$ we have $\H^{j+1}_{\et}(P,\G) \simeq \H^{j+1}_{\het}(P^h,\G^h)$.

Writing $\mc{N}$ for $\bigoplus_{\ell=0}^r \O_P(n_\ell)$, the long exact sequence of cohomology gives us a commutative diagram with exact rows

\begin{centering}  
\begin{tikzcd}[column sep=1.1em]
\H^j_{\et}(P,\mc{K}) \arrow[r]\arrow[d] &
\H^j_{\et}\left(P, \mc{N}\right) \arrow[r]\arrow[Isom,d] &
 \H^j_{\et}(P,\F) \arrow[r]\arrow[d] &
 \H^{j+1}_{\et}(P,\mc{K}) \arrow[r]\arrow[Isom,d] &
 \H^{j+1}_{\et}\left(P, \mc{N}\right) \arrow[Isom,d]\\
\H^j_{\het}(P^h,\mc{K}^h) \arrow[r] & \H^j_{\het}\left(P^h, \mc{N}^h \right)\arrow[r] & \H^j_{\het}(P^h,\F^h) \arrow[r]& \H^{j+1}_{\het}(P^h,\mc{K}^h) \arrow[r] & \H^{j+1}_{\het}\left(P^h, \mc{N} ^h\right)
\end{tikzcd} 

\end{centering}

where we note that $\mc{N}=\bigoplus_{\ell=0}^r \O_P(n_\ell)$ satisfies \shfcm{P} 
and $\H^{j+1}_{\et}(P,\mc{K}) \simeq \H^{j+1}_{\het}(P^h,\mc{K}^h)$ by the inductive hypothesis on $j+1$. By the first Four Lemma, the middle arrow $\H^j_{\et}(P,\F) \to \H^j_{\het}(P^h,\F^h)$ is surjective for arbitrary coherent $\F$. Hence by considering a similar commutative diagram corresponding to a presentation of $\mc{K}$, the map $\H^j_{\et}(P,\mc{K}) \to \H^j_{\het}(P^h,\mc{K}^h)$ is also surjective. Therefore we can use the second Four Lemma to get injectivity of the middle arrow; hence the middle arrow is an isomorphism.

Therefore all coherent sheaves $\F$ on $P$ satisfy \shfcm{P} if $\O_P$ does. \end{proof}

To show that $\O_{\P^1}$ satisfies \shfcm{\P^1}, we first check in degree $0$.

\begin{proposition}\label{prop:h0p1} Let $(A,I)$  be a Henselian pair with $A$ Noetherian. Let $P=\P_A^1$ be the projective line over $A$, with $P^h$ its $I$-adic Henselization. Then the comparison map $A=\H_{\et}^0(P,\O_P) \to \H^0_{\het}(P^h,\O_{P^h})$ is an isomorphism.
\end{proposition}
\begin{proof} In order to compute $\H^0_{\het}(P^h,\O_{P^h})$, we can describe $P$ as the union of two affine opens $U = \spec(A[t])$ and $V=\spec(A[1/t])$, with $U \cap V = \spec(A[t,1/t])$. We will write $A\{t\}$ for the $I$-adic Henselization $(A[t])^h$ of the polynomial ring $A[t]$, and similarly $A\{1/t\}=(A[1/t])^h, A\{t,1/t\}=(A[t,1/t])^h$ for the $I$-adic Henselizations of $A[1/t],A[t,1/t]$.

We note that the map $A\{t\} \to (A[t])^\wedge$ is faithfully flat, hence injective, by \spcite{0AGV}{Lemma}. Similarly we have injections $A\{1/t\} \hookrightarrow (A[1/t])^\wedge$ and $A\{t,1/t\} \hookrightarrow (A[t,1/t])^\wedge.$ Elements of the completion $(A[t])^\wedge$ have the form $\sum_{n=0}^\infty a_nt^n$ such that for any $N \ge 1$, there exists $M$ so that $a_n \in I^N$ for $n \ge M$. Elements of the other completions are described similarly as series in $t^{-1}$ or two-sided power series in $t$ with coefficients tending $I$-adically to $0$ as the exponents go to $\pm \infty$. 

Consider the map $A\{t\} \times A\{1/t\} \to A\{t,1/t\}$ given by $(f,g) \mapsto  f-g$ (we write $f,g$ also for their respective images in $A\{t,1/t\}$), and the similarly defined map $(A[t])^\wedge \times (A[1/t])^\wedge \to (A[t,1/t])^\wedge$. The kernel of the first map is $\H^0_{\het}(P^h,\O_{P^h})$ by the sheaf condition; it's clear that $A \subset \H^0_{\het}(P^h,\O_{P^h})$ (via the diagonal map $A \to A\{t\} \times A\{1/t\}$). 

Also, we see that $\H^0_{\het}(P^h,\O_{P^h})$ is contained in the kernel of $(A[t])^\wedge \times (A[1/t])^\wedge \to (A[t,1/t])^\wedge$, which we can see is $A$ using the explicit description of elements of these completions given above. Then $\H^0_{\het}(P^h,\O_{P^h}) = A = \H^0_{\et}(P,\O_P)$ as we desired to show.\end{proof}

To prove that the structure sheaf on $P=\P_A^1$ for $A$ a Henselian $\mathbf{F}_p$-algebra satisfies GHGA comparison, it remains to show that $\H^i_{\het}(P^h,\O_{P^h})=0$ for $i \ge 1$. Since $A$ is a Henselian $\mathbf{F}_p$-algebra, by Theorem~\ref{thm:hetzarcoh} it suffices to show that the {\it Zariski} cohomology $\H^i(P^h,\O_{P^h})$ vanishes.

The projective line $P=\P_A^1$ is covered by the two affine opens $U=\spec(A[t]), V = \spec(A[1/t])$. Since we have ``Theorem B'' for affine Henselian schemes in positive characteristic (Theorem~\ref{thm:thmbpos}), we can use the two-term \v{C}ech complex $A\{t\} \times A\{1/t\} \to A\{t,1/t\}$ to compute the cohomology $\H^i(P^h,\O_{P^h})$; it follows that $\H^i_{\het}(P^h,\O_{P^h})=0$ for $i > 1$.

In order to compute $\H^1(P^h,\O_{P^h})$ with the \v{C}ech complex $A\{t\} \times A\{1/t\} \to A\{t,1/t\}$, we will identify the $I$-adic Henselization as the subring of elements of the $I$-adic completion which are ``algebraic'' over the base. This will be shown using approximation for Henselian pairs with the aid of a $G$-ring hypothesis in order to leverage a powerful theorem of Popescu:

\begin{theorem}[Popescu]\label{thm:pop} A regular homomorphism of Noetherian rings is a filtered colimit of smooth ring maps.
\end{theorem} 
\begin{proof} See \spcite{07GC}{Theorem}. Swan gives an exposition of Popescu's proof of this theorem in \cite{swan}.\end{proof}

More specifically, we will use Artin-Popescu approximation for Henselian pairs, which is a consequence of Popescu's Theorem. This is \spcite{0AH5}{Lemma}, which we state here without proof.

\begin{theorem}[Artin-Popescu approximation]\label{thm:artpop} Let $(A,I)$ be a Noetherian Henselian pair, with $A^\wedge$ the $I$-adic completion of $A$. Assume $(A,I)$ is the Henselization of a pair $(B,J)$ for $B$ a Noetherian $G$-ring.

Then given $f_1, \ldots , f_ m \in A[x_1, \ldots , x_ n]$ and $\widehat{a}_1, \dots , \widehat{a}_ n \in A^\wedge$ such that  $f_ j(\widehat{a}_1, \dots , \widehat{a}_ n) = 0$ for all $j$, there exists for every $N \ge 1$ elements $a_1,\dots,a_n \in A$ such that $\widehat{a_i}-a_i \in I^N$ and such that $f_j(a_1,\dots,a_n)=0$ for all $j$. 
\end{theorem}

We can now identify the Henselization with the ``algebraic'' subring of the completion, in the case of a Noetherian domain and $G$-ring.

\begin{lemma}\label{lem:hensalgcplt} Let $B$ be a Noetherian domain and $G$-ring, and $J \subset B$ an ideal. Let $B^h$ be the Henselization of $B$ along $J$, and $B^\wedge$ the $J$-adic completion of $B$. Then the map $B^h \to B^\wedge$ is injective, and an element $f \in B^\wedge$ satisfies some nonzero polynomial in $B[x]$ if and only if $f$ lies in the image of $B^h$. 
\end{lemma}

\begin{proof}
The map $B^h \to B^\wedge$ is faithfully flat, therefore it is injective, so $B^h \subseteq B^\wedge$. Elements of $B^h$ are clearly algebraic over $B$. 

Now we assume that we have an element $f \in  B^\wedge$ and a nonzero polynomial $q(x) \in B[x]$ such that $q(f)=0$. Then since $B$ is a $G$-ring, by Theorem~\ref{thm:artpop} for any integer $N \ge 1$ we can find $g_N \in B^h$ with $f-g_N \in J^NB^\wedge$ and $q(g_N)=0$.  This gives us a sequence $\{g_N\}_{N=1}^\infty$ of elements of $B^h$ converging $I$-adically to $f$ in $B^\wedge$.

Since $B$ is a Noetherian domain, it injects into its fraction field $L$. The tensor product $B^h \tens_B L$ is also a finite product of fields $\prod_i L_i$ by \spcite{0AH1}{Lemma}, and $B^h \to B^h \tens_B L$ is injective because $B \to B^h$ is flat. 

There are only finitely many roots of $q$ in $\prod_i L_i$, so there are only finitely many roots of $q$ in $B^h \hookrightarrow \prod_i L_i$. Therefore we have a subsequence $\{g_{N_r}\}_{r=1}^\infty$ of  $\{g_N\}_{N=1}^\infty$ which are all equal to a single element $g_0 \in B^h$ with $q(g_0)=0$, so $f = g_0 \in B^h$ as we desired to show.\end{proof}

Using Lemma~\ref{lem:hensalgcplt}, we will be able to reduce our computations with the \v{Cech} complex above to the following lemma:

\begin{lemma}\label{lem:alglseries} Let $A$ be a Noetherian ring  such that $A$ has characteristic $p > 0$, and $I \subset A$ an ideal. Then for $f \in (A[t,1/t])^\wedge$ (the $I$-adic completion) which is algebraic over $A[t,1/t]$ we can write $f=f_++f_-$ such that $f_+ \in (A[t])^\wedge, f_- \in (A[1/t])^\wedge$ with $f_+$ algebraic over $A[t]$ and $f_-$ algebraic over $A[1/t]$. 
\end{lemma}
\begin{proof}  This is shown by de Jong in the Stacks Project Blog post at \cite{seriesblog}; we provide the proof here for the reader's convenience.

The element $f$ is a two-sided power series in $t$ with coefficients going $I$-adically to $0$ as the exponent goes to $\pm \infty$. We separate $f$ into $f_++f_-$ with $f_+$ a power series in $t$ and $f_-$ a power series in $1/t=t^{-1}$. It suffices to show that $f_+$ is algebraic over $A[t]$, by the symmetry of $t$ and $1/t$. 

 We can assume the constant term of $f$ is part of $f_+$. If $f$ is algebraic over $A[t,1/t]$, we have a relation $\sum_{i=0}^m P_if^i=0$ with $P_i \in A[t,1/t], m > 0, P_m \ne 0$. By multiplying by an appropriate power of $t$ we can assume that the $P_i$ are in $A[t]$.

Let $A(t)$ denote the fraction field of $A[t]$. We see by the above that the $A(t)$-span of the powers of $f$ is finite-dimensional. Therefore the $A(t)$-span of $\{f^{p^i}\}_{i=0}^\infty$ is also finite-dimensional, meaning we have a relation $\sum_{i=0}^r Q_if^{p^i}$ with $Q_i \in A(t), r>0, Q_r \ne 0$. By clearing denominators we can assume that the $Q_i$ are polynomials in $t$. 

Now because $p=0 \in A$ we know that $\left(\sum_{i=0}^r Q_if_+^{p^i}\right)+\left(\sum_{i=0}^r Q_if_-^{p^i}\right) = 0.$ We can consider the coefficients of large positive powers of $t$ on the left side of this equation; in particular, for very large $m$ we know that the coefficient of $t^m$ in $\sum_{i=0}^r Q_if_-^{p^i}$ is $0$ by our choice of $f_-$. Hence in order for this equation to hold, for $m$ sufficiently large, the coefficient of $t^m$ in $\sum_{i=0}^r Q_if_+^{p^i}$ must also be $0$. 

Therefore $\sum_{i=0}^r Q_if_+^{p^i}$ is actually a polynomial in $t$. Let $\sum_{i=0}^r Q_if_+^{p^i}=:Q \in A[t]$. As we assumed $Q_r \ne 0$, we see that $Q-\sum_{i=0}^r Q_if_+^{p^i}=0$ is a nonzero polynomial relation, so $f_+$ is algebraic over $A[t]$ as desired.\end{proof}

We can now compute $\H^1(P^h,\O_{P^h})$  for $P=\P_A^1$ using the \v{Cech} complex. By Corollary~\ref{cor:hetzarcmpsn}, from now on we can work with the Zariski topology for the purposes of GHGA comparison, and we will do so for the remainder of this section.

\begin{lemma}\label{lem:h1gen} Let $(A,I)$ be a Noetherian Henselian pair such that $A$ has characteristic $p > 0$.  Let $P=\P_A^1$ be the projective line over $A$, with $P^h$ its $I$-adic Henselization. Then $\O_P$ satisfies \shfcm{P}.
\end{lemma}

We will prove the lemma by a series of reductions, beginning with the case that \begin{enumerate}[(a)] \item $A$ is a Noetherian domain and $G$-ring, then \item assuming $A$ is a Noetherian normal reduced $G$-ring and reducing to case (a), and finally \item reducing the case of a general Noetherian Henselian pair to case (b). \end{enumerate}

\begin{proof}By Proposition~\ref{prop:h0p1} we have an isomorphism $\H^0_{\et}(P,\O_P) \simeq \H^0_{\het}(P^h,\O_{P^h})$. We also have isomorphisms $\H^j_{\et}(P,\O_P) \simeq \H^j_{\het}(P^h,\O_{P^h}) = 0 $ for $j \ge 2$ by Lemma~\ref{lem:thmbhet} and \spcite{03F7}{Lemma} (see the discussion in the proof of Proposition~\ref{prop:p1red}). 

Therefore in order to show that $\O_P$ satisfies \shfcm{P}, 
 it suffices to show that $\H^1_{\et}(P,\O_P) \simeq \H^1_{\het}(P^h,\O_{P^h})$ via the natural map. By Corollary~\ref{cor:hetzarcmpsn}, in fact it is enough to show that the {\it Zariski} comparison map $\H^1(P,\O_P) \to \H^1(P^h,\O_{P^h})$ is an isomorphism, or equivalently that $\H^1(P^h,\O_{P^h})=0$, in order to prove that $\O_P$ satisfies \shfcm{P}. 
 
 Using \spcite{01ET}{Lemma} and applying Theorem~\ref{thm:thmbpos}, we know we can compute $\H^1(P^h,\O_{P^h})$ with \v{C}ech cohomology associated to the Henselization of the affine cover of $P$ given by the two affine opens $U = \spec(A[t])$ and $V=\spec(A[1/t])$, with intersection $U\, \cap \,V = \spec(A[t,1/t])$. Thus, $\H^1(P,\O_P) \to \H^1(P^h,\O_{P^h})$ is an isomorphism if and only if the map $A\{t\} \times A\{1/t\} \to A\{t,1/t\}$ given by $(f,g) \mapsto  f-g$ is surjective.\footnote{As in the proof of Proposition~\ref{prop:h0p1}, we write $A\{t\}$ for the $I$-adic Henselization $(A[t])^h$ of the polynomial ring $A[t]$, and similarly $A\{1/t\}=(A[1/t])^h, A\{t,1/t\}=(A[t,1/t])^h$ for the $I$-adic Henselizations of $A[1/t],A[t,1/t]$.}

 \begin{enumerate}[(a)] \item We assume that $A$ is a Noetherian $\mathbf{F}_p$-algebra domain and $G$-ring. 

Since $A$ is a Noetherian domain and $G$-ring, the same is true for $A[t]$ and $A[t,1/t]$ by \spcite{07PV}{Proposition}. Therefore we can apply Lemma~\ref{lem:hensalgcplt} to $B=A[t],A[1/t]$ and $A[t,1/t]$. We then see that showing that the map $A\{t\} \times A\{1/t\} \to A\{t,1/t\}$ is surjective amounts to showing that for $f \in (A[t,1/t])^\wedge$ (the $I$-adic completion) which is algebraic over $A[t,1/t]$, we can write $f$ as a sum $f=f_++f_-$ such that $f_+ \in (A[t])^\wedge, f_- \in (A[1/t])^\wedge$ with $f_+$ algebraic over $A[t]$ and $f_-$ algebraic over $A[1/t]$. This can be done by Lemma~\ref{lem:alglseries}, so $\H^1(P^h,\O_{P^h})=0$. It follows that $\O_P$ satisfies \shfcm{P} for $A$ which is a domain and $G$-ring, as discussed above.

\item Now assume that $A$ is a Noetherian normal reduced $G$-ring. It follows that $A$ is a finite product of normal domains $\prod_i A_i$ by \spcite{030C}{Lemma}. Therefore $\spec(A)$ is the disjoint union of the $\spec(A_i)$, and $P$ is the disjoint union of $\P_{A_i}^1=:P_i$. These decompositions into components carry over to the Henselizations and $P^h$ is the disjoint union of $P_i^h$. Thus $\H^1(P^h,\O_{P^h})=\bigoplus_i \H^1(P_i^h,\O_{P_i^h})$.

The $A_i$ are clearly Noetherian domains  and $G$-rings,  
so by the previous case (a), each $\H^1(P_i^h,\O_{P_i^h})$ is $0$, so $\H^1(P^h,\O_{P^h})=0$ and $\O_P$ satisfies \shfcm{P} in this case as well.
 
 \item Finally we treat the general Noetherian case. By hypothesis, $A$ is an $\mathbf{F}_p$-algebra. We can write $A$ as the filtered direct limit of its finitely generated subalgebras $A = \colim A_\lambda$. Let $I_\lambda = I \cap A_\lambda$, and $A_\lambda^h$ be the Henselization of $A_\lambda$ along $I_\lambda$. We see that $A=\colim A_\lambda^h, I = \colim I_\lambda^h$ by \spcite{0A04}{Lemma}.

Since $A_\lambda$ for each $\lambda$ is a finitely generated $\mathbf{F}_p$-algebra, we see that $A_\lambda^h$ is the quotient of the Henselization (along some ideal) of a polynomial ring $\mathbf{F}_p[t_1,\dots,t_r]$. Note that the Henselization of $\mathbf{F}_p[t_1,\dots,t_r]$ is a Noetherian normal reduced $G$-ring. (See \cite[\href{https://stacks.math.columbia.edu/tag/033B}{Lemma 033B}, \href{https://stacks.math.columbia.edu/tag/0AH3}{Lemma 0AH3}, \href{https://stacks.math.columbia.edu/tag/033C}{Lemma 033C}, \href{https://stacks.math.columbia.edu/tag/037D}{Lemma 037D}, \href{https://stacks.math.columbia.edu/tag/0AGV}{Lemma 0AGV}]{stackp}.)

Therefore each pair $(A_\lambda^h,I_\lambda^h)$ is the quotient of some Henselian pair $(A_\lambda',I_\lambda')$ with $A_\lambda'$ a Noetherian normal reduced $G$-ring. Let $\smash{P_\lambda:=\P_{A_\lambda^h}^1, P_\lambda' := \P_{A_\lambda'}^1}$. 

By the previous case (b), for all $\lambda$ the  structure sheaf $\O_{P_\lambda'}$ satisfies \shfcm{P_\lambda'}; then by Proposition~\ref{prop:p1red}, in fact $P_\lambda'$ satisfies \cmpsn. 

Because $\spec(A_\lambda^h) \hookrightarrow \spec(A_\lambda')$ is a closed immersion, so is $\smash{P_\lambda\hookrightarrow P_\lambda'.}$ Closed immersions are finite, so we can apply Lemma~\ref{lem:pipush}---the pushforward of $\O_{P_\lambda}$ is coherent, so $\O_{P_\lambda}$ satisfies \shfcm{P_\lambda}. 

Therefore for all $\lambda$ the map $A_\lambda^h\{t\} \times A_\lambda^h\{1/t\}  \to A_\lambda^h\{t,1/t\} $ is surjective. Since Henselization commutes with direct limits \spcite{0A04}{Lemma}, we see that $\colim A_\lambda^h\{t\} = A\{t\}$. Similarly, we see that  $\colim A_\lambda^h\{1/t\} = A\{1/t\}$ and  $\colim A_\lambda^h\{t,1/t\} = A\{t,1/t\}$.

Then $A\{t\} \times A\{1/t\} \to A\{t,1/t\}$ is the colimit of surjective maps, so it is surjective. Therefore $\H^1(P^h,\O_{P^h})=0$.\end{enumerate}
Thus for $(A,I)$ a general Noetherian Henselian pair in characteristic $p>0$, we have shown that $\H^1(P^h,\O_{P^h})=0$, so $\O_P$ satisfies \shfcm{P} as we desired to show.  \end{proof}



%
%
%

%

\begin{theorem}\label{thm:poscharprop} Let $(A,I)$ be a Noetherian Henselian pair such that $A$ has characteristic $p > 0$, and let $X \to \spec(A)$ be a proper $A$-scheme. Then $X$ satisfies \cmpsn.
\end{theorem}
\begin{proof}
This follows from Theorem~\ref{thm:propred}, Proposition~\ref{prop:p1red}, and Lemma~\ref{lem:h1gen}, because if $B$ is an $A$-algebra with $A \to B$ essentially finite-type, then $B$ and $B^h$ have positive characteristic if $A$ does. \end{proof}

\subsection{Non-Noetherian comparison}\label{sec:nonnoeth}

In this section we will extend Theorem~\ref{thm:poscharprop} to the case of a base ring which is not necessarily Noetherian, but still has characteristic $p > 0$. (Recall that, via Corollary~\ref{cor:hetzarcmpsn}, we will now work with Zariski cohomologies for the purpose of GHGA comparison.)

\thmposcharnonnoeth

In the Appendix it is shown in Lemmas~\ref{lem:seplimcohom}~and~\ref{lem:henslimcohom} that limits of schemes or of Henselian schemes are compatible with cohomology in order to reduce the non-Noetherian case to the Noetherian one. We will use these lemmas to prove Theorem~\ref{thm:poscharnonnoeth}.

\begin{proof}[Proof of Theorem~\ref{thm:poscharnonnoeth}]
Fix $X \to \spec(A)$ proper and $\F$ a quasi-coherent sheaf on $X$. We will show that $\F$ satisfies \shfcm{X}.

For $X \to \spec(A)$ proper, we can find a closed immersion $\iota: X \hookrightarrow X'$ with $X'$ proper and finitely presented over $A$. (See, for example, \spcite{09ZR}{Lemma}). 

Now by Lemma~\ref{lem:pipush}, we may reduce to the situation of $X \to \spec(A)$ both proper and finitely presented. By \spcite{01PJ}{Lemma}, $\F$ can be written as the filtered colimit of finitely presented, quasi-coherent $\O_X$-modules, so we can reduce to the case of $\F$ finitely presented using \spcite{01FF}{Lemma}. 

Write $A$ as the filtered direct limit of its subalgebras $A_i$ which are finitely generated over the finite field $\mathbf{F}_p$. Set $I_i=I \cap A_i$. Then clearly $\colim I_i=I$. Hence by the universal property of Henselization, $A=\colim A_i^h,I=\colim I_i^h$ for $(A_i^h,I_i^h)$ the Henselization of the pair $(A_i,I_i)$ as Henselization commutes with filtered colimits. Furthermore the $A_i^h$ are Noetherian by \spcite{0AGV}{Lemma}.

Since we are now assuming that $X$ is finitely presented over $A$, there exists $i_0$ and a finitely presented $A_{i_0}^h$-scheme $X_{i_0}$ such that $X=X_{i_0} \tens_{A_{i_0}^h} A$ by \spcite{01ZM}{Lemma}. We may assume that $X_{i_0}$ is proper over $A_{i_0}^h$ by \spcite{081F}{Lemma}. 

Furthermore, as $\F$ is finitely presented, we can increase $i_0$ as necessary so that there exists a coherent sheaf $\F_{i_0}$ on $X_{i_0}$ so that $\F=\F_{i_0} \times_{X_{i_0}} X$ by \spcite{01ZR}{Lemma}.

Setting $X_i=X_{i_0} \tens_{A_{i_0}^h} A_i^h$ for $A_{i_0} \subset A_i$ and defining $\F_i$ similarly, if we let $Y_i$ be the closed subscheme of $X_i$ corresponding to $I_i^h$, we see that we are in the situations of Lemmas~\ref{lem:seplimcohom}~and~\ref{lem:henslimcohom}. Therefore we see that $\colim_i \H^j(X_i,\F_i) \simeq \H^j(X,\F)$ and $\colim_i \H^j(X_i^h,\F_i^h) \simeq \H^j(X^h,\F^h)$ for all $j \ge 0$. Because each $A_i^h$ is a Noetherian $\mathbf{F}_p$-algebra, we know that the natural map $\H^j(X_i,\F_i) \to \H^j(X_i^h,\F_i^h)$ is an isomorphism for all $i$ and all $j \ge 0$ by Theorem~\ref{thm:poscharprop}, so $\F$ also satisfies \shfcm{X} as we desired to show.\end{proof}

As mentioned in Remark~\ref{rmk:relp1hypo}, we will generalize the notion of ``relative comparison'' from Proposition~\ref{prop:relp1} to the non-Noetherian setting and general proper maps in the following theorem.

\begin{theorem}\label{thm:nonnoethrelcmpsn} Let $(A,I)$ be a Henselian pair such that $A$ has characteristic $p > 0$, and $f: Y \to X$ a proper map of $A$-schemes.

Assume that any quasi-coherent sheaf on $X$ satisfies \shfcm{X} and that $X$ is finitely presented over $\spec(A)$. Then any quasi-coherent sheaf on $Y$ satisfies \shfcm{Y}. 
\end{theorem}

To prove Theorem~\ref{thm:nonnoethrelcmpsn}, we will generalize Proposition~\ref{prop:genhighdirim} 
to the case of a general proper morphism of finite-type schemes over a Noetherian base. 

\begin{lemma}\label{lem:henshighdirim2}
Let $(A,I)$ be a Noetherian Henselian pair such that $A$ has characteristic $p > 0$. Let $X$ be a finite-type $A$-scheme and $f: Y \to X$ be a proper morphism of schemes. For $\F$ a coherent sheaf on $Y$ and any $j \ge 0$, the map $(R^jf_{\et,*}\F_{\et})^h \to R^jf^h_{\het,*}(\F^h)$ of sheaves on $(X^h)_{\het}$ arising from the base change map of Lemma~\ref{lem:basechangemap} is an isomorphism.\end{lemma} 
\begin{proof} Since $f$ is proper and $X$ is locally Noetherian, the higher direct image sheaves $R^jf_{\et,*}\F_\et$ are coherent on $X$.
Therefore the pullback $(R^jf_{\et,*}\F_{\et})^h$ to $(X^h)_\het$ is a coherent sheaf of $\O_{(X^h)_{\het}}$-modules. 

We proceed as in the proof of Proposition~\ref{prop:genhighdirim} to show the base change map is an isomorphism by checking on stalks;  
fix a point $x \in X^h$, and choose a geometric point $\ol{x}$ of $X$ lying over $x$, with $\ol{x}^h$ the corresponding geometric point of $X^h$. Recall the equations computing the stalks:\vspace{-.25in}

\begin{align}
\tag{\ref{eqn:stalk1}}
((R^jf_{\et,*}\F_{\et})^h)_{\ol{x}^h} &=  \colim_{(U,\ol{u})} \H^j_{\et}(U \times_X Y,\F),\\
\tag{\ref{eqn:stalk2}}
(R^jf^h_{\het,*}(\F^h)_{\het})_{\ol{x}^h} &= \colim_{(U,\ol{u})} \H^j_{\het}(U^h \times_{X^h} Y^h,\F^h),
\end{align} 

where $(U,\ol{u}) \to (X,\ol{x})$ are affine \Et neighborhoods. The remainder of the argument to show that both stalks can be identified with  $\H^j_{\et}(Y \times X_{(\ol{x})},\F|_{Y \times X_{(\ol{x})}})$ via the natural maps is almost exactly as in the proof of Proposition~\ref{prop:genhighdirim}. The only difference in the proof is how we show that for $U=\spec(B), U^h=\sph(B^h)$, the comparison map $\H^j_{\et}((U \times_X Y)_{B^h},\F_{B^h}) \to \H^j_{\het}(U^h \times_{X^h} Y^h,\F^h)$ is an isomorphism, which we now describe.

%

Note that the $B^h$-scheme $U':=(U \times_X Y)_{B^h}$ is proper over $B^h$ since $f: Y \to X$ is proper. Furthermore we see that $U^h \times_{X^h} Y^h$ is isomorphic to the $I$-adic Henselization $(U')^h$.

Because $A$ is an $\Fp$-algebra and  $B$ is a finite-type $A$-algebra, both $B,B^h$ are Noetherian $\Fp$-algebras \spcite{0AGV}{Lemma}. Therefore the proper $B^h$-scheme $U'$ satisfies \cmpsn by Theorem~\ref{thm:poscharprop}; it follows from Corollary~\ref{cor:hetzarcmpsn} that the comparison map $$\H^j_{\et}((U \times_X Y)_{B^h},\F_{B^h}) =\H^j_{\et}(U',\F_{U'})\to \H^j_{\het}((U')^h,\F^h)=\H^j_{\het}(U^h \times_{X^h} Y^h,\F^h)$$ is an isomorphism.

The rest is precisely as the proof of Proposition~\ref{prop:genhighdirim}. \end{proof}

We can now prove Theorem~\ref{thm:nonnoethrelcmpsn}. 

\begin{proof}[Proof of Theorem~\ref{thm:nonnoethrelcmpsn}]
Fix a quasi-coherent sheaf $\F$ on $Y$ and let $Y^h$ be the $I$-adic Henselization of $Y$ (which is a Henselian scheme over $\sph(A)$). We will show that the canonical \hetale cohomology comparison map $\H^n_{\et}(Y,\F) \to \H^n_{\het}(Y^h,\F^h)$ is an isomorphism for all $n$, so $\F$ satisfies \shfcm{Y}. 

Since $f: Y \to X$ is proper, we can find a closed immersion $\iota: Y \hookrightarrow Y'$ with $Y'$ proper and finitely presented over $X$. (See, for example, \spcite{09ZR}{Lemma}). 

Now by Lemma~\ref{lem:pipush}, $\F$ satisfies \shfcm{Y} if and only if $\iota_*\F$ satisfies \shfcm{Y'}. Thus we may reduce to the situation where $f$ is both proper and finitely presented. By \spcite{01PJ}{Lemma}, $\F$ can be written as the filtered colimit of finitely presented, quasi-coherent $\O_Y$-modules, so we can reduce to the case of $\F$ finitely presented using \spcite{01FF}{Lemma}. 

As in the proof of Theorem~\ref{thm:poscharnonnoeth}, we can write the pair $(A,I)$ as the filtered direct limit of $(A_i, I_i:=I \cap A_i)$ for the finitely generated $\Fp$-subalgebras $A_i$ of $A$. In fact $A=\colim A_i^h,I =\colim I_i^h$ for $(A_i^h,I_i^h)$ the Henselization of the pair $(A_i,I_i)$, and the $A_i^h$ are Noetherian by \spcite{0AGV}{Lemma}.

Since we are now assuming that $f: Y \to X$ is a morphism of finitely presented $A$-schemes, there exists $i_0$ and a morphism of finitely presented $A_{i_0}^h$-schemes $f_{i_0}: Y_{i_0} \to X_{i_0}$ such that $X=X_{i_0} \tens_{A_{i_0}^h} A, Y=Y_{i_0} \tens_{A_{i_0}^h} A$ and $f$ is the base change of $f_{i_0}$ by \spcite{01ZM}{Lemma}. We may assume that $f_{i_0}$ is proper by \spcite{081F}{Lemma}. 

Furthermore, as $\F$ is finitely presented, we can increase $i_0$ as necessary so that there exists a coherent sheaf $\F_{i_0}$ on $Y_{i_0}$ so that $\F=\F_{i_0} \times_{Y_{i_0}} Y$ by \spcite{01ZR}{Lemma}.

Set $X_i=X_{i_0} \tens_{A_{i_0}^h} A_i^h$ for $A_{i_0} \subset A_i$ and define $Y_i, f_i, \F_i$ similarly. Then for all $i \ge i_0$ and all $j \ge 0$, the map $(R^j(f_i)_{\et,*}(\F_i)_{\et})^h \to R^j(f_i^h)_{\het,*}(\F_i^h)$ of sheaves on $(X_i^h)_{\het}$ arising from the base change map of Lemma~\ref{lem:basechangemap} is an isomorphism by Lemma~\ref{lem:henshighdirim2}. 

Fix $j$ and consider for all $i \ge i_0$ the pullbacks of the sheaves $R^j(f_i)_{\et,*}(\F_i)_{\et}$ from $(X_i)_{\et}$ to $X_\et$. The limit of these pullbacks is the sheaf $R^jf_{\et,*}\F_\et$ by \cite[Corollary 5.9.6]{leifu}. It follows that the limit of the pullbacks of $(R^j(f_i)_{\et,*}(\F_i)_{\et})^h$ to $(X^h)_\het$ is $(R^jf_{\et,*}\F_\et)^h$ (where $X_i^h$ is the Henselization of $X_i$ along the closed subscheme $X_i'$ corresponding to $I_i^h$; consider Remark~\ref{rmk:morhet}). 

A similar argument (again implicitly using the equivalence of $(Y_i^h)_\het$ and $(Y_i')_\et$ for $Y_i'$ the closed subscheme of $Y_i$ corresponding to $I_i^h$ and Remark~\ref{rmk:morhet}) finds that the limit of the pullbacks of $R^j(f_i^h)_{\het,*}(\F_i^h)$ to $X_\et$ is $R^jf^h_{\het,*}(\F^h)$. The maps $(R^j(f_i)_{\et,*}(\F_i)_{\et})^h \to R^j(f_i^h)_{\het,*}(\F_i^h)$ from Lemma~\ref{lem:basechangemap} are compatible with these limits, so we have an isomorphism $(R^jf_{\et,*}\F_\et)^h \to R^jf^h_{\het,*}(\F^h)_\het$ for each $j$. 

To complete the proof, we will compare the Leray spectral sequence $\H^m_{\et}(X,R^nf_*\F) \implies \H^{m+n}_{\et}(Y,\F)$, and the analogous spectral sequence for $f^h$ and $\F^h$. (We use $m,n$ to avoid confusion with the indices $i$ above.)

As in the proof of Proposition~\ref{prop:relp1}, we choose an injective resolution $\mc{I}^\bullet$ of $\F_{\et}$ by $\O_{Y_{\et}}$-modules, with which we can compute the \Et cohomology of $\F$. Similarly, we choose an injective resolution $\mc{J}^\bullet$ of $\F^h_{\het}$ by $\O_{(Y^h)_{\het}}$-modules to compute the \hetale cohomology of $\F^h$. 

We may pull back $\mc{I}^\bullet$ to an exact sequence of $\O_{(Y^h)_{\het}}$-modules (with degree $m$ term $(\mc{I}^m)^h$), which maps to $\mc{J}^\bullet$ since $\mc{J}^\bullet$ is an injective resolution. This gives us a map from the pullback of $\mc{I}^\bullet$ to $(Y^h)_{\het}$ to $\mc{J}^\bullet$, which can be lifted through the steps of constructing the Leray spectral sequences to give us a map of spectral sequences that must respect the filtration on limit terms.

By hypothesis, the map $\H^m_{\et}(X,R^nf_*\F) \to \H^m_{\het}(X^h,(R^nf_*\F)^h)$ is an isomorphism, and we showed above that $(R^nf_{\et,*}\F_\et)^h \to R^nf^h_{\het,*}(\F^h)_\het$ for all $n$. Thus pullback to the Henselization gives an isomorphism on the second sheet
$$\H^m_{\et}(X,R^nf_*\F) \to \H^m_{\het}(X^h,(R^nf_*\F)^h) \simeq \H^m_{\het}(X^h,R^nf^h_*\F^h)$$ for our map of spectral sequences Since this map respects the filtration on limit terms, the comparison map $\H^n_{\et}(Y,\F) \to \H^n_{\het}(Y^h,\F^h)$ is an isomorphism for all $n$. It follows that $\F$ satisfies \shfcm{Y} by Corollary~\ref{cor:hetzarcmpsn}. \end{proof}


\section{Algebraization of coherent sheaves}\label{chap:algzation}

In this section we consider a different GHGA problem: algebraizability of sheaves on a Henselization.

\begin{definition}\label{def:algbility}
For a proper and finitely presented morphism of schemes $X \to \spec(A)$ with $(A,I)$ a Henselian pair and a finitely presented sheaf of $\O_{X^h}$-modules $\G$ (for $X^h$ the $I$-adic Henselization of $X$), we say that $\G$ is {\bf algebraizable} if there exists a finitely presented sheaf $\F$ on $X$ with $\F^h \simeq \G$. 
\end{definition}

We restrict ourselves to the case of a proper scheme in order to leverage the fact that for a proper scheme $X$ over a pair $(A,I)$ such that $A$ is Noetherian and $I$-adically complete, pullback of coherent sheaves to the Henselization is an exact and fully faithful functor (see Proposition~\ref{prop:hensffulexact}). Furthermore, regardless of characteristic, a proper scheme satisfies \cmpsn in degree $0$ under a completeness condition \hfcite{\lemhogen}. We will prove that coherent subsheaves of algebraizable coherent sheaves are algebraizable (Theorem~\ref{thm:subshf}), and deduce the consequences mentioned in Section~\ref{sub:results}.

\begin{enumerate}[(I)]
\item Our proof of algebraizability of \emph{subsheaves} begins by reducing to the case where the base ring $A$ is a Noetherian $G$-ring, using the fact that a Henselian ring $A$ is the filtered colimit of the Henselizations of its finitely generated $\Z$-subalgebras.
\item When the base $A$ is a Noetherian $G$-ring, we can use Popescu's theorem (Theorem~\ref{thm:pop}) to reduce further to the case where the base ring is complete.
\item By leveraging formal GAGA results we can prove Theorem~\ref{thm:subshf} in the case of a complete base ring, from which a Henselian version of Chow's theorem (Corollary~\ref{cor:henschow}) and algebraizability of maps between Henselian schemes or maps of coherent sheaves (Corollaries~\ref{cor:essimg},~\ref{cor:hensmapsalg}) will follow.
\item Finally we use a counterexample of de Jong \cite{cohpropblog} to show that even in the case of the projective line in characteristic 0 (or in mixed characteristic) we already have failures of algebraizability for abstract coherent sheaves on $X^h$ (which are not necessarily subsheaves of some algebraizable sheaf $\F^h$), so algebraizability for all coherent sheaves on a proper and finitely presented $A$-scheme remains uncertain only with a positive characteristic hypothesis.
\end{enumerate}

\subsection{Noetherian $G$-rings}\label{sec:noethg}

We first reduce algebraizability of a general finitely presented sheaf of $\O_{X^h}$-modules to the situation where $A$ is a Noetherian $G$-ring.

\begin{lemma}\label{lem:noethg}
Let $(A,I)$ be a Henselian pair; $X,X'$ and $Y$ proper and finitely presented $A$-schemes with morphisms $X \to X', X^h \to Y^h$ over $\spec(A), \sph(A)$ respectively; and $\G \to \G'$ a morphism of finitely presented sheaves of $\O_{X^h}$-modules. 

Then there exists a Noetherian $G$-ring $A_0 \subseteq A$ and a map of Henselian pairs $(A_0,I_0) \to (A,I)$, $X_0,X_0',Y_0$ proper and finitely presented $A_0$-schemes with morphisms $X_0 \to X_0', (X_0)^h \to (Y_0)^h$ over $\spec(A_0), \sph(A_0)$ respectively, and $\G_0 \to \G_0'$ a morphism of finitely presented sheaves of $\O_{(X_0)^h}$-modules such that: \begin{enumerate}[(i)] \item $X_0 \times_{A_0} A = X, X_0' \times_{A_0} A = X'$, and the map $X \to X'$ arises from the map $X_0 \to X_0'$;
\item $Y_0 \times_{A_0} A=Y$, and the map $X^h \to Y^h$ arises from the map $(X_0)^h \to (Y_0)^h$;
\item the pullback of the map $\G_0\to \G_0'$ along the map $X^h \to X_0^h$ is the map $\G \to \G'$.
 \end{enumerate}
 \end{lemma}
\begin{proof}
As in the proof of Theorem~\ref{thm:poscharnonnoeth}, we can write $A$ as the filtered direct limit of its subalgebras $A_i$ which are finitely generated over $\mathbf{Z}$. Let $I_i=I \cap A_i$. If $(A_i^h,I_i^h)$ is the Henselization of $(A_i,I_i)$ for each $i$, we have $A= \colim A_i^h$. Because $X$ is finitely presented, for some $i_0$ we have $X=X_{i_0} \times_{A_{i_0}^h} A$ for some proper morphism $X_{i_0} \to \spec(A_{i_0}^h)$, and similarly for $X',Y$. 

By \spcite{01ZM}{Lemma}, we can assume that the map $X \to X'$ arises from some map $X_{i_0 } \to X_{i_0}'$ for some index $i_0$. Similarly, using Lemma~\ref{lem:schemelem}, we can assume (possibly after increasing $i_0$) that the map $X^h \to Y^h$ arises from some map $X_{i_0}^h \to Y_{i_0}^h$. Note that $X_{i_0},X_{i_0}',Y_{i_0}$ are proper and finitely presented $A_{i_0}^h$-schemes.

In the same way, we can (again, possibly after further increasing $i_0$) use Lemma~\ref{lem:sheaflem} to find finitely presented $\O_{X_{i_0}^h}$-modules $\G_{i_0}, \G_{i_0}'$  with a map $\G_{i_0} \to \G_{i_0}'$ which pulls back to $\G \to \G'$ via the map $X^h \to X_{i_0}^h$.  

Since $A_{i_0}^h$ has been chosen to be the Henselization of a finitely generated $\Z$-algebra, we see that $A_{i_0}^h$ is a Noetherian $G$-ring by \spcite{0AH3}{Lemma}. Then letting $A_0=A_{i_0}^h,I_0=I_{i_0}^h$, we get the desired result.
\end{proof}

\begin{remark}\label{rmk:noethg} In the setting of Lemma~\ref{lem:noethg}, if $\G_{0}$ is algebraizable---meaning we have a coherent sheaf $\F_{0}$ on $X_{0}$ such that $\F_{0}^h \simeq \G_{0}$---then we see immediately that $\G$ is algebraizable, since $\F^h \simeq \G$ for $\F$ the pullback of $\F_{0}$ along the map $X \to X_{0}$. For similar reasons, if the map $\G_0 \to \G_0'$ arises from a map of coherent sheaves $\F_0 \to \F_0'$ on $X_0$, it is clear that the map $\G \to \G'$ is algebraizable; and if the map $X_0^h \to Y_0^h$ arises from a map $X_0 \to Y_0$ over $A_0$, then the map $X^h \to Y^h$ arises from a map $X \to Y$ over $A^h$.
\end{remark}

Therefore we have reduced the general algebraizability problem for finitely presented proper schemes to the case where the base is a Noetherian $G$-ring. We next reduce to the case where the base ring $A$ is $I$-adically complete.

\subsection{Reduction to complete case}\label{sec:cplt}

\begin{lemma}\label{lem:cplt}
Let $(A,I)$ be a Henselian pair  with $A$ a Noetherian $G$-ring, and let $X$ be a proper $A$-scheme and $\G$ a coherent sheaf on the $I$-adic Henselization $X^h$ of $X$. Let $X'=X \times_A A^\wedge$ and $\G'$ be the pullback of $\G$ along the map $(X')^h \to X^h$, which is a coherent sheaf.

If there exists a coherent sheaf $\F'$ on $X'$ such that $\G' \simeq (\F')^h$---i.e., if $\G'$ is algebraizable---then $\G$ is algebraizable.
\end{lemma}

\begin{proof} Since $A$ is a Noetherian $G$-ring, the map $A \to A^\wedge$ from $A$ to its $I$-adic completion is regular by \spcite{0AH2}{Lemma}. We can now leverage Popescu's theorem (Theorem~\ref{thm:pop}) to see that $A \to A^\wedge$ is a filtered colimit of smooth ring maps. Furthermore, the map $A \to A^\wedge$ is faithfully flat by \spcite{0AGV}{Lemma}. Then $X' \to X$ must be flat and surjective since $\spec(A^\wedge) \to \spec(A)$ is.

We then have a commutative cube (\ref{cube1}) below, with left and right faces Cartesian. The zigzag arrows represent pullback of sheaves, such as with the leftmost zigzag arrow showing that $\G'$ is the pullback of $\G$. 

We would like to find a coherent sheaf $\F$ on $X$ such that $\F^h \simeq \G$. In order to do so, we will use Popescu's theorem. Since $A \to A^\wedge$ is a filtered colimit of smooth ring maps, there exists a smooth $A$-algebra $C_0$ such that $\F'$ is the pullback of a coherent sheaf $\F_0$ on the proper \linebreak$C_0$-scheme $X_{C_0}=X \times_A C_0$. Let $\G_0 := \G \times_{X^h} X_{C_0}^h$; we then (possibly replacing $C_0$ by some further smooth $A$-algebra in the colimit yielding $A^\wedge$) can use Lemma~\ref{lem:sheaflem} to arrange that $\F_0^h \simeq \G_0$ descending the isomorphism $(\F')^h \simeq \G'$. 

\begin{equation}\label{cube1}\tag{$\dagger$}
\tikzcdset{arrow style=tikz, diagrams={>={Computer Modern Rightarrow[width=10pt,length=6pt]}}}\begin{tikzcd}
\spec(A) &[-3pt]&[-35pt]& \arrow[lll]\arrow[ddd,leftarrow]\arrow[ddr,leftarrow]\sph(A) &[-3pt]&[-35pt] \\[-5pt]
&&\F&&&\G\arrow[ddd,squiggly,bend left=10]\arrow[lll,leftarrow,"\text{\Large?}" description,squiggly,crossing over,bend right=7]\\[-30pt]
&X\arrow[luu]&&& \arrow[lll,crossing over]\arrow[ddd,leftarrow]X^h&\\\arrow[uuu]\spec(A^\wedge) &&& \arrow[lll]\arrow[rdd,leftarrow]\sph(A^\wedge) && \\[-5pt]
&&\F'&&&\G'\arrow[lll,leftarrow,squiggly,crossing over,bend right=7]\\[-25pt]
& \arrow[uuu,crossing over]X'\arrow[luu] &&& \arrow[lll,crossing over](X')^h&&\\
\end{tikzcd}
\end{equation} \vspace{-.2in}

This gives us a larger commutative diagram (\ref{cube2}) below, again with all left and right faces Cartesian. The arrow from $\F_0$ to $\F$ represents that we will construct $\F$ as a pullback of $\F_0$ along the dashed arrow from $X$ to $X_{C_0}$, which is a section of the morphism $X_{C_0} \to X$. 

\begin{equation}\label{cube2}\tag{$\dagger\dagger$}
\tikzcdset{arrow style=tikz, diagrams={>={Computer Modern Rightarrow[width=10pt,length=6pt]}}}\begin{tikzcd}
\spec(A) &[-3pt]&[-35pt]& \arrow[lll]\arrow[ddd,leftarrow]\arrow[ddr,leftarrow]\sph(A) &[-3pt]&[-35pt] \\[-5pt]
&&\F&&&\G\arrow[dddddd,squiggly,bend left,end anchor={[xshift=.2ex,yshift=-1.2ex]}]\arrow[ddd,squiggly,bend left=10]\arrow[lll,leftarrow,"\text{\Large?}" description,squiggly,crossing over,bend right=7]\\[-20pt]
&X\arrow[luu]&&& \arrow[lll,crossing over]\arrow[ddd,leftarrow]X^h&\\
\arrow[uuu]\spec(C_0) &&& \arrow[lll]\arrow[ddd,leftarrow]\arrow[ddr,leftarrow]\sph(C_0^h) && \\[-5pt]
&&\F_0\arrow[uuu,squiggly,bend right,crossing over]&&&\G_0\arrow[lll,leftarrow,squiggly,crossing over,bend right=7]\arrow[ddd,squiggly,bend left=10]\\[-20pt]
& \arrow[uuu,crossing over]\arrow[uuu,crossing over,leftarrow,bend left,dashed,end anchor={[xshift=-.6ex,yshift=.2ex]}]\arrow[luu]X_{C_0} && &\arrow[lll,crossing over]\arrow[ddd,leftarrow]X_{C_0}^h&\\[-5pt]
\arrow[uuu]\spec(A^\wedge) &&& \arrow[lll]\arrow[rdd,leftarrow]\sph(A^\wedge) && \\
&&\F'\arrow[uuu,squiggly,leftarrow,crossing over,bend right]&&&\G'\arrow[lll,leftarrow,squiggly,crossing over,bend right=7]\\[-20pt]
& \arrow[uuu,crossing over]X'\arrow[luu] &&& \arrow[lll,crossing over](X')^h&&\\
\end{tikzcd}
\end{equation} \vspace{-.2in}

The map $C_0 \to A^\wedge$ gives us a map $C_0 \to A/I$. We can find an \Et $A$-algebra $A'$ with $A'/IA' \simeq A/I$ and a map $C_0 \to A'$ lifting the map $C_0 \to A/I$ by \spcite{07M7}{Lemma}. Because $A$ is Henselian, the map $A \to A'$ has a section lifting the isomorphism $A/I \simeq A'/IA'$, so we have a {\it section} $C_0 \to A$ over $A$ lifting the map $C_0 \to A/I$.

This map $C_0 \to A$ gives us the desired dashed arrow $X \to X_{C_0}$ over $X$. As in (\ref{cube2}), let $\F$ be the pullback of $\F_0$ along this map. Now we wish to show that $\F^h$ is isomorphic to $\G$, which is represented by the zigzag arrow with a question mark in (\ref{cube2}).

We recall that we have an isomorphism $\F_0^h \simto \G_0$, represented by the second horizontal zigzag arrow of (\ref{cube2}). We can pull back this map along the map $X_{C_0}^h \to X^h$,  giving us an isomorphism $\F^h \simto \G$ (here we use that $X^h \to X_{C_0}^h \to X^h$ is the identity map). Hence $\G$ is algebraizable, as we desired to show.\end{proof}

\begin{remark}\label{rmk:cplt} Similarly to Lemma~\ref{lem:noethg}, Lemma~\ref{lem:cplt} can also be used to reduce algebraizability of maps of coherent sheaves or of maps of the Henselizations of proper schemes over a Noetherian $G$-ring Henselian base $A$ to the case of a Noetherian complete base. 

The key point is that by Popescu's theorem, $A \to A^\wedge$ is a filtered colimit of smooth ring maps. Hence for any object(s) over $A^\wedge$ coming from a finite set of finitely presented data, we can find a smooth $A$-algebra $C_0$ with a map $C_0 \to A^\wedge$ such that our objects over $A^\wedge$ are the pullbacks of objects over $C_0$. We can then use the Henselian property of $A$ to get a section $C_0 \to A$ along which we can pullback algebraizations over $A^\wedge$ to algebraizations over $A$.

Although it is not always the case that $A^\wedge$ inherits the Noetherian $G$-ring property from $A$ (see \cite[Section 5]{gringnishi} for a counterexample), we will not need to use the $G$-ring property when proving algebraizability of subsheaves over a complete base $A^\wedge$, instead relying on the completeness property to leverage formal GAGA results.
\end{remark}


\subsection{Subsheaves of algebraizable sheaves}\label{sec:algsub}

In this section we show that subsheaves of algebraizable sheaves are algebraizable. 

\thmsubshf

This is proved by F. Kato in \cite[Theorem 5]{kato17} when $A$ is a valuation ring which is Henselian with respect to the principal ideal $(a)$, for $a$ a nonzero element of the maximal ideal of the valuation ring. Our proof is similar, but does not use the notions of $I$-adically adhesive or $I$-adically universally adhesive which are used in \cite{kato17}.
\begin{proof}
Applying Lemma~\ref{lem:noethg} to the the injective morphism $\G \hookrightarrow \F^h$ of finitely presented $\O_{X^h}$-submodules, we can obtain a map of Henselian pairs $(A_0,I_0) \to (A,I)$ with $A_0$ a Noetherian $G$-ring, as well as \begin{itemize} \item  a proper and finitely presented $A_{0}$-scheme $X_{0}$ for which $X=X_{0} \tens_{A_{0}} A$; \item  a finitely presented sheaf of $\O_{X_{0}}$-modules $\F_{0}$ for which $\F = \F_{0} \times_{X_{0}} X$; \item and a morphism of finitely presented sheaves of $\O_{{X_0}^h}$-modules $\G_{0} \to \F_0^h$ which pulls back to the injective morphism $\G \hookrightarrow \F^h$ by the map $X^h \to X_0^h$. \end{itemize}

Since pullback is right exact, we can replace $\G_{0}$ with its sheafified image in $\F_{0}^h$, as that will also pull back to $\G$. Thus we can identify $\G_{0}$ with a coherent subsheaf of $\F_{0}^h$. 

As in Remark~\ref{rmk:noethg}, if $\G_0$ is algebraizable as the Henselization of some coherent subsheaf of $\F_0$, then $\G$ is also algebraizable (as the Henselization of a coherent subsheaf of $\F$). Thus replacing $A$ with $A_0$, we have reduced to the case where $A$ is a Noetherian $G$-ring. By Lemma~\ref{lem:cplt} and Remark~\ref{rmk:cplt}, we can furthermore assume that $A$ is $I$-adically complete (and we will no longer need to assume that $A$ is a $G$-ring). 

Consider the formal scheme $X^\wedge \to \spf(A)$ which is the $I$-adic completion of $X$, and is proper over $\spf(A)$. The map $X^\wedge \to X$ factors through $X^h$.

 We write $\F^\wedge,\G^\wedge$ for the pullbacks of $\F,\G$ along the maps $X^\wedge \to X, X^\wedge \to X^h$ respectively. Then clearly $\G^\wedge \subset \F^\wedge$, so by formal GAGA \cite[Theorem 5.1.4]{ega31} we have a coherent sheaf $\G_1 \subset \F$ on $X$ such that $\G_1^\wedge =(\G_1^h)^\wedge \simeq \G^\wedge$ as subsheaves of $\F^\wedge$.
 
To show that the finitely presented sheaves $\G_1^h$ and $\G$ are equal as subsheaves of $\F^h$, it suffices to work over Henselian-affine opens of $X^h$ arising from affine opens of $X$. Since $\G_1^h,\G$ are finitely presented, we can consider a new affine setting: $\spec(B^h)$ for $B$ a finite-type $A$-algebra (with $I$-adic Henselization and completion $B^h, B^\wedge$ respectively) and $\F^h=\widetilde{M}$, $\G=\widetilde{N}$, $\G_1^h=\widetilde{N'}$ for $M$ a finitely generated $B^h$-module with finitely generated submodules $N,N'$: we wish to show that if  $N \tens_{B^h} B^\wedge $ and $N' \tens_{B^h} B^\wedge$ are equal as $B^\wedge$-submodules of $M \tens_{B^h} B^\wedge$ then $N$ and $N'$ are equal as $B^h$-submodules of $M$.
 
Because $B$ is finite-type over $A$, which is Noetherian, we see that $B^h \to B^\wedge$ is faithfully flat. Thus, since the inclusions $N \subset N+N', N' \subset N+N'$ inside $M$ become equalities after tensoring with $B^\wedge$, they are equalities of $B^h$-submodules as we desired to show.\end{proof}

Chow's theorem \cite[Theorem V]{chowthm} states that a closed analytic subspace of complex projective space is an algebraic subvariety. We can deduce a Henselian version of this from Theorem~\ref{thm:subshf}.

\corhenschow 

\begin{proof} If $A$ is Noetherian, then the closed immersion $Y \subseteq X^h$ is defined by a finitely presented ideal sheaf inside $\O_{X^h}$. This ideal sheaf is algebraizable by Theorem~\ref{thm:subshf}, so there is a finitely presented ideal of $\O_X$ defining a finitely presented closed subscheme $Z \subseteq X$ with $Z^h = Y$. This completes the proof of the Noetherian case.

If $A$ is not Noetherian, then as in the proof of Theorem~\ref{thm:poscharnonnoeth} and Lemma~\ref{lem:noethg}, we can write $A$ as the filtered direct limit of the Henselizations of its finitely generated $\Z$-subalgebras $A_i$; in other words, if $(A_i^h,I_i^h)$ is the Henselization of $(A_i,I \cap A_i)$ for each $i$, we have $A= \colim A_i^h$. The $A_i^h$ are Noetherian by \spcite{0AGV}{Lemma}.

Hence we can apply Lemma~\ref{lem:subschemelem} to find an index $i_0$ and a finitely presented closed Henselian subscheme $Y_{i_0} \subseteq X_{i_0}^h$ such that $X_{i_0} \otimes_{A_{i_0}^h} A = X, Y_{i_0} \otimes_{A_{i_0}^h} A = Y$. 

Then because $A_{i_0}^h$ is Noetherian \spcite{0AGV}{Lemma} we see that $Y_{i_0} \subseteq X_{i_0}^h$ is algebraizable---meaning there exists a finitely presented closed subscheme $Z_{i_0} \subseteq X_{i_0}$ with $Z_{i_0}^h=Y_{i_0}$. Letting $Z=Z_{i_0} \otimes_{A_{i_0}^h} A \subseteq X$, we see that $Z^h=Y$, algebraizing $Y$. 
\end{proof}

As discussed in Section~\ref{sub:results}, we can also extend \hfcite{\prophensffulexact} (stated here as Proposition~\ref{prop:hensffulexact}), while relaxing its completeness hypothesis:

\coressimg

\begin{proof} We know this functor is exact and faithful for any $A$ (and in fact for the Henselization of any scheme along any closed subscheme) by \hfcite{\lemhensfaithexact}. The essential image is closed under subobjects by Theorem~\ref{thm:subshf}, and it is closed under quotients as well since the functor is exact. 

For fullness, fix some coherent sheaves $\F,\G$ on $X$ and assume we have a morphism $\F^h \to \G^h$ on $X^h$. The graph $\mc{H}$ of this morphism is a subsheaf of $\F^h \oplus \G^h$, so by Theorem~\ref{thm:subshf} we have a subsheaf $\mc{H}_0$ of $\F \oplus \G$ such that $\mc{H}_0^h=\mc{H}$ inside $\F^h \oplus \G^h$. 

We know that the map $$\mc{H} \hookrightarrow \F^h \oplus \G^h \to \F^h$$ is an isomorphism. Consider the map $$\mc{H}_0 \hookrightarrow \F \oplus \G \to \F.$$ Since the Henselization functor is exact, if $\mc{K}=\ker(\mc{H}_0 \to \F)$ we know that $\mc{K}^h$ is the kernel of $\mc{H} \to \F^h$, so $\mc{K}^h$ vanishes. Thus because the Henselization functor is faithful, it follows that $\mc{K}$ also vanishes. Applying a similar argument to the cokernel of $\mc{H}_0 \hookrightarrow \F \oplus \G \to \F$, we see that this map is also an isomorphism. 

Therefore $\mc{H}_0$ defines a morphism of sheaves $\F \to \G$ ``algebraizing'' the morphism $\F^h \to \G^h$. This proves fullness of the functor as well, completing the proof.
\end{proof}

We can now show that maps between the Henselizations of proper and finitely presented $A$-schemes are algebraizable.

\corhensmapsalg

 As with Theorem~\ref{thm:subshf}, F. Kato proves this corollary in \cite[Corollary 6]{kato17} when $A$ is a valuation ring which is Henselian with respect to a principal ideal $(a)$, for $a$ a nonzero element of the maximal ideal of the valuation ring. Our proof is similar, but we will go into more detail.

\begin{proof}By Lemma~\ref{lem:noethg}, for any morphism $X \to Y$ we can find a map of Henselian pairs $(A_0,I_0) \to (A,I)$ with $A_0$ a Noetherian $G$-ring for which $X=X_{0} \tens_{A_0} A, Y=Y_{0} \tens_{A_0} A$ for proper and finitely presented $A_0$-schemes $X_{0},Y_{0}$ such that $X \to Y$ is the base change of a morphism $X_{0} \to Y_{0}$. 

Similarly, for any morphism $X^h \to Y^h$, we can find $A_0,X_{0},Y_{0}$ as above so that $X^h \to Y^h$ is the base change of a morphism $X_{0}^h \to Y_{0}^h$ using Lemma~\ref{lem:schemelem}. This allows us to reduce to the case where $A$ is Noetherian and a $G$-ring. We can further reduce to the case where $A$ is also $I$-adically complete by Lemma~\ref{lem:cplt} and Remark~\ref{rmk:cplt}. 

Let $X^\wedge,Y^\wedge$ be the proper formal schemes over $\spf(A)$ which are the $I$-adic completions of $X,Y$ respectively. By formal GAGA, we have a bijection $$\Hom_{\spec(A)}(X,Y) \to \Hom_{\spf(A)}(X^\wedge,Y^\wedge)$$ between the Hom set of morphisms of $A$-schemes $X \to Y$ and the Hom set of morphisms of $A$-formal schemes $X^\wedge \to Y^\wedge$. This bijection factors through the map $\Hom_{\spec(A)}(X,Y) \to \Hom_{\sph(A)}(X^h,Y^h)$. Hence $\Hom_{\spec(A)}(X,Y) \to \Hom_{\sph(A)}(X^h,Y^h)$ must be injective. To show it is surjective, we have to ``descend'' an arbitrary map $X^h \to Y^h$ over $\sph(A)$ to a map $X \to Y$ over $\spec(A)$.

For a morphism $g: X^h \to Y^h$, we can consider its graph $G \subset X^h \times_{\sph(A)} Y^h=(X \times_{\spec(A)} Y)^h$. By Corollary~\ref{cor:henschow}, we have a closed subscheme $F$ of $X \times Y$ such that $F^h = G$ as closed Henselian subschemes of $X^h \times_{\sph(A)} Y^h$.

To show that $F$ defines the graph of a morphism $X \to Y$, we want the morphism $$F \hookrightarrow X \times_{\spec(A)} Y \to X$$ to be an isomorphism. We already know that the map  $F^h  \hookrightarrow X^h \times_{\sph(A)} Y^h \to X^h$ is equal to $G \to X^h$, which is an isomorphism. Now we will use the following lemma:

\begin{lemma}\label{lem:lastlem} For $(A,I)$ a Noetherian Henselian pair and $f: T \to S$ a morphism of proper $A$-schemes, $f$ is an isomorphism if $f^h$ is.
\end{lemma}

\begin{proof}[Proof of Lemma]
If $f^h: T^h \to S^h$ is an isomorphism, then for all $r$, the morphism $f_r: T_r \to S_r$ of the proper $A/I^{r+1}$-schemes $T_r := T \times_A A/I^{r+1}, S_r := S \times_A A/I^{r+1}$ is an isomorphism. Therefore $f^\wedge: T^\wedge \to S^\wedge$, the $I$-adic completion of $f$, is an isomorphism.

Hence for any point $t \in T_0$, the morphism $(f^\wedge)_t$ is an isomorphism. Therefore $f$ is flat and quasi-finite at all points $t$ of $T_0$.  

The flat locus and the quasi-finite locus of $f$ are open by \cite[Corollary 13.1.4]{ega43} and \cite[Theorem 11.3.1]{ega43} respectively. 

Therefore the locus of points of $T$ at which $f$ is both flat and quasi-finite is open and contains $T_0$. However, since $I$ is contained in the Jacobson radical $\Jac(A)$, we see that the only open set in $\spec(A)$ containing $\spec(A/I)$ is $\spec(A)$ itself; because the morphism $T \to \spec(A)$ is proper, we see that the only open set in $T$ containing $T_0= T \times_A A/I$ is all of $T$. 

Therefore $f$ is flat and quasi-finite. Since $T$ is proper over $\spec(A)$, $f$ is in fact finite and flat. The function sending points $t \in T$ to the fiber degree (the dimension of $f^{-1}(f(t))$) is locally constant. Since the fiber degree for all $t \in T_0$ is $1$ (because $f_0$ is an isomorphism), we see that there exists an open set $U \supset T_0$ on which the fiber degree is $1$. However, since the only open set containing $T_0$ is all of $T$, we see that $f$ is a finite flat morphism of fiber degree $1$. Hence $f$ is an isomorphism as we desired to show.\end{proof}

By Lemma~\ref{lem:lastlem}, since the composite map $$F^h \simto G \hookrightarrow X^h \times_{\sph(A)} Y^h \to X^h$$ is an isomorphism, the composite map $F \hookrightarrow X \times_{\spec(A)} Y \to X$ is an isomorphism as well. Therefore $F$ defines the graph of a morphism $X \to Y$ which algebraizes our morphism $X^h \to Y^h$. \end{proof}

\subsection{Failures of algebraizability}

The question remains open whether, in positive characteristic, \emph{all} coherent sheaves on a proper and finitely presented scheme over a Henselian base are algebraizable. However, in characteristic $0$ or mixed characteristic, we cannot have algebraizability for all coherent sheaves on a proper and finitely presented $A$-scheme. In fact even in the case of the projective line over a complete DVR, there exists of a locally free coherent sheaf of finite rank that is not algebraizable. Our example has rank $2$; we are unsure if this is minimal. 

\begin{example}\label{ex:algfail} For $A$ a complete DVR with fraction field $K$ of characteristic $0$, there exists a locally free coherent sheaf of rank $2$ on $(\P_A^1)^h$ which is not algebraizable; this follows from the nonvanishing of $\H^1((\P_A^1)^h,\O_{(\P_A^1)^h})$ that was proved by de Jong in \cite{cohpropblog}. We reproduce that argument here for the convenience of the reader.

We will make use of \v{Cech} cohomology, which always injects into cohomology for $\H^1$. Thus to show nonvanishing of $H^1((\P_A^1)^h,\O_{(\P_A^1)^h})$, it will suffice to show nonvanishing of the first \v{Cech} cohomology for $\O_{(\P_A^1)^h}$ with respect to the Henselization of standard covering of $\P_A^1$. As in the positive characteristic case, this is the cokernel of the map $A\{t\} \times A\{1/t\} \to A\{t,1/t\}$. \footnote{As in the proof of Proposition~\ref{prop:h0p1}, $A\{t\}=(A[t])^h$ the $I$-adic Henselization of the polynomial ring $A[t]$, and similarly $A\{1/t\}=(A[1/t])^h, A\{t,1/t\}=(A[t,1/t])^h$ for the $I$-adic Henselizations of $A[1/t],A[t,1/t]$.}

Since a complete DVR is a Noetherian domain and G-ring, the same is true for $A[t], A[t,1/t]$ by \spcite{07PV}{Proposition}. 

Thus we can proceed as in the proof of Lemma~\ref{lem:h1gen}, applying Lemma~\ref{lem:hensalgcplt} to $B=A[t],A[1/t]$ and $A[t,1/t]$. Each element $f \in (A[t,1/t])^\wedge$ (the $I$-adic completion) is a two-sided power series in $t$ with coefficients going $I$-adically to $0$ as the exponent goes to $\pm \infty$, and $f$ can be written as $f=f_++f_-$ for $f_+ \in (A[t])^\wedge, f_- \in (A[1/t])^\wedge$ (so $f_+$ is a power series in $t$ and $f_-$ is a power series in $t^{-1}$). Showing that the map $A\{t\} \times A\{1/t\} \to A\{t,1/t\}$ has nonzero cokernel amounts to showing that for some $f =f_++f_- \in (A[t,1/t])^\wedge$ which is algebraic over $A[t,1/t]$, the element $f_+ \in (A[t])^\wedge$ is {\it not} algebraic over $A[t]$.

Let $\omega$ be a uniformizer of the complete DVR $A$. In \cite{seriesblog} de Jong proves that in the above situation, the element $f=\sqrt{(1+\omega t)(1+\omega/t)} \in (A[t,1/t])^\wedge$, visibly algebraic over $A[t,1/t]$, has $f_+$  not algebraic over $A[t]$. Therefore the first \v{Cech} cohomology, and so the first cohomology, for $\O_{(\P_A^1)^h}$ does not vanish: $H^1((\P_A^1)^h,\O_{(\P_A^1)^h}) \ne 0$.

It will follow that there exists a locally free coherent sheaf of rank $2$ on $(\P_A^1)^h$ which is not algebraizable; to show this, we proceed with a similar argument to \hfcite{\exgrecowrong}. 

Since the group $\H^1((\P_A^1)^h,\O_{(\P_A^1)^h})$ describes extensions of $\O_{(\P_A^1)^h}$ by itself \spcite{0B39}{Lemma}, we have a sheaf $\mc{E}$ on $(\P_A^1)^h$ which fits into a short exact sequence $\O_{(\P_A^1)^h} \hookrightarrow \mc{E} \twoheadrightarrow \O_{(\P_A^1)^h}$ that does not split. Clearly $\mc{E}$ is coherent, even locally free of rank $2$.

Assume $\mc{E}$ is algebraizable, i.e. there exists a coherent sheaf $\mc{E}'$ on $\P_A^1$ such that $(\mc{E}')^h \simeq \mc{E}$. By Proposition~\ref{prop:hensffulexact}, Henselization of coherent sheaves is an exact and fully faithful functor. Hence the maps of the short exact sequence $\O_{(\P_A^1)^h} \hookrightarrow \mc{E} \twoheadrightarrow \O_{(\P_A^1)^h}$ also algebraize to a diagram $\O_{\P_A^1} \hookrightarrow \mc{E}' \twoheadrightarrow \O_{\P_A^1}$ that is short exact since $(\P_A^1)^h \to \P_A^1$ is flat over all closed points and $\P_A^1$ is proper over $A$. This necessarily splits since $\H^1(\P_A^1,\O_{\P_A^1}) = 0$. However, by the functoriality of Henselization of coherent sheaves, that would give a map $\mc{E} \to  \O_{(\P_A^1)^h}$ making $\mc{E}$ a split extension of $\O_{(\P_A^1)^h}$ by $\O_{(\P_A^1)^h}$, a contradiction. \end{example}

\section*{Appendix: Limits}
\addcontentsline{toc}{section}{Appendix: Limits}
\stepcounter{section}
\renewcommand{\thesection}{A}

In this Appendix we collect lemmas on limits that are useful for reducing from a non-Noetherian setting to the Noetherian setting. 

The following standard fact, which we record primarily as motivation for the Henselian version in the next lemma, is proved via the \v{Cech}-to-derived spectral sequence.

\begin{lemma}\label{lem:seplimcohom} Consider a cofiltered inverse system $\{X_\alpha\}$ of quasi-compact and separated schemes with affine transition morphisms $\phi_{\alpha\beta}: X_\alpha \to X_\beta$. Let $X=\varprojlim X_\alpha$ and let $\phi_\alpha: X \to X_\alpha$ be the canonical map.

Furthermore take quasi-coherent sheaves $\F_\alpha$ on $X_\alpha$ which form a filtered direct system via compatible isomorphisms $\phi_{\alpha\beta}^*\F_\beta \simto \F_\alpha$ for $\alpha \ge \beta$, and set $\F=\colim \phi_\alpha^*\F_\alpha$. Then the natural map $$\colim \H^j(X_\alpha,\F_\alpha) \to \H^j(X,\F)$$ is an isomorphism for each $j \ge 0$.
\end{lemma}

Lemmas~\ref{lem:seplimcohom}~and~\ref{lem:henslimcohom} are used in Section~\ref{sec:nonnoeth} to prove Henselian cohomology comparison in the case of a non-Noetherian Henselian base ring (Theorem~\ref{thm:poscharnonnoeth}).

\begin{lemma}\label{lem:henslimcohom} 
Consider a cofiltered inverse system $\{X_\alpha\}$ of quasi-compact and separated schemes with affine transition morphisms $\phi_{\alpha\beta}: X_\alpha \to X_\beta$, with a compatible system of closed subschemes $Y_\alpha \subset X_\alpha$ (i.e. $\phi_{\alpha\beta}^{-1}(Y_\beta)=Y_\alpha$). Let $X=\varprojlim X_\alpha, Y= \varprojlim Y_\alpha$, so $Y$ is a closed subscheme of $X$. Let $\phi_\alpha: X \to X_\alpha$ be the canonical map.

Furthermore take quasi-coherent sheaves $\F_\alpha$ on $X_\alpha$ which form a filtered direct system via isomorphisms $\phi_{\alpha\beta}^*\F_\beta \simto \F_\alpha$ for $\alpha \ge \beta$, and set $\F=\colim \phi_\alpha^*\F_\alpha$. Then if $X_\alpha^h$ is the Henselization of $X_\alpha$ along $Y_\alpha$ and $X^h$ is the Henselization of $X$ along $Y$, the natural map $$\colim \H^j(X_\alpha^h,\F_\alpha^h) \to \H^j(X^h,\F^h)$$ is an isomorphism for each $j \ge 0$.
\end{lemma}
\begin{proof}
We begin by reducing to the case $i=0$.

Taking the limit of the \v{C}ech-to-cohomology spectral sequence (as in \spcite{01ES}{Lemma}) for all open coverings of $X_\alpha^h$ we have for each $\alpha$ a spectral sequence $\prescript{\alpha}{}E_\bullet^\dublet$ with second sheet $\check{H}^i(X_\alpha^h,\underline{\H}^j(\F_\alpha^h))$ which abuts to $\H^{i+j}(X_\alpha^h,\F_\alpha^h)$. We also get a similar spectral sequence for $X^h$ with second sheet $\check{H}^i(X^h,\underline{\H}^j(\F^h))$ which abuts to $\H^{i+j}(X^h,\F^h)$. 

By the functoriality of the spectral sequences for \v{C}ech cohomology, we have a commutative diagram

\begin{center}
\begin{tikzcd}[column sep=2em]
\colim_{\alpha} \check{H}^i(X_\alpha^h,\underline{\H}^j(\F_\alpha^h)) \arrow[Rightarrow,r] \arrow[d] &\colim_\alpha \H^{i+j}(X_\alpha^h,\F_\alpha^h) \arrow[d]\\
\check{H}^i(X^h,\underline{\H}^j(\F^h)) \arrow[Rightarrow,r]& \H^{i+j}(X^h,\F^h) 
\end{tikzcd}
\end{center}

where the horizontal arrows are the abutments of the spectral sequences. Let $E_2^{i,j}=\check{H}^i(X^h,\underline{\H}^j(\F^h))$.

For $j > 0$, the presheaves $\underline{\H}^j(\F_\alpha^h)$ on $X_\alpha^h$ and $\underline{\H}^j(\F^h)$ on $X^h$ sheafify to $0$ by \spcite{01E3}{Lemma}. Therefore for all $\alpha$ and all $q > 0$, the term $\prescript{\alpha}{}E_2^{0,q}$ vanishes.

The limit term $\H^{q}(X_\alpha^h,\F_\alpha^h)$ of the spectral sequence $\prescript{\alpha}{}E_\bullet^\dublet$ depends only on: $\prescript{\alpha}{}E_2^{0,q}$ (which is $0$ for $q > 0$), the terms $\prescript{\alpha}{}E_2^{i,j}$ with $j < q$, and maps among these, because of the direction of the differential maps. 

Therefore if we fix $q > 0$ and assume that for all $j < q$ and all $X_\alpha,Y_\alpha,$ etc. we have the statement of the lemma, we can use the commutative diagram above to deduce that the lemma is true for $q$ as well. 

Therefore we have reduced to showing that $\colim \Gamma(X_\alpha^h,\F_\alpha^h) \to \Gamma(X^h,\F^h)$ is an isomorphism. We may easily reduce to the case of all of the $X_\alpha$ being affine.

Now write $X_\alpha=\spec(A_\alpha)$ and $Y_\alpha=\spec(A_\alpha/I_\alpha)$ for rings $A_\alpha$ and ideals $I_\alpha \subset A_\alpha$. Then if \linebreak$A=\colim A_\alpha, I = \colim I_\alpha$, we have $X=\spec(A),Y=\spec(A/I)$. Furthermore since the $\F_\alpha$ are quasi-coherent we see that they correspond to $A_\alpha$-modules $M_\alpha$ with $\F$ corresponding to the $A$-module $M=\colim M_\alpha$. By \hfcite{\lemmodpullback}, for each sheaf $\F_{\alpha}^h$ on $X_\alpha^h$ we have $$\F_{\alpha}^h=\widetilde{M_\alpha \tens_{A_\alpha} A_\alpha^h},$$ and similarly on $X^h$ we have an equality of sheaves $$\F^h = \widetilde{M \tens_A A^h}.$$ Then what we wish to show is that $\colim_\alpha M_\alpha \tens_{A_\alpha} A_\alpha^h \to M \tens_A A^h$ is an isomorphism, where $A_\alpha^h$ is the Henselization of $A_\alpha$ along $I_\alpha$ and $A^h$ is the Henselization of $A$ along $I$. This is true since Henselization commutes with filtered colimits \spcite{0A04}{Lemma}. \end{proof}

The following three lemmas are used in Section~\ref{chap:algzation}. Lemmas~\ref{lem:schemelem}~and~\ref{lem:sheaflem} are used to reduce algebraization of subsheaves (Theorem~\ref{thm:subshf}) to the case of a Noetherian G-ring base (Lemma~\ref{lem:noethg}), and then to the case of an $I$-adically complete base ring (Lemma~\ref{lem:cplt}); Lemma~\ref{lem:subschemelem} is used in the proof of the Henselian Chow's Theorem (Corollary~\ref{cor:henschow}.)

\begin{lemma}\label{lem:schemelem} Suppose that $(A,I)$ is a filtered colimit of Henselian pairs $(A_i,I_i)$. (Note that from the definitions it is clear that $(A,I)$ is also Henselian.) If $X,Y$ are finitely presented $A$-schemes and $f: X^h \to Y^h$ is a map of Henselian schemes over $\sph(A)$, then there exists an index $i_0$ and finitely presented $A_{i_0}$-schemes $X_{i_0}, Y_{i_0}$ with a map of $I_{i_0}$-adic Henselizations $f_{i_0}: X_{i_0}^h \to Y_{i_0}^h$ over $\sph(A_{i_0})$ so that \begin{enumerate}[(i)] \item $X=X_{i_0} \otimes_{A_{i_0}} A, Y = Y_{i_0} \otimes_{A_{i_0}} A$; \item and $f$ is the pullback of $f_{i_0}$ to a map $X^h \to Y^h$.  \end{enumerate}
\end{lemma}
\begin{proof} Locally, the map $f: X^h \to Y^h$ of Henselian schemes, which are both \hlfp over $\sph(A)$, can be described by a map $\sph(C) \to \sph(B)$, which arises from a map $\varphi: B \to C$ where $B,C$ are \hfp $A$-algebras. Then $B,C$ are the $I$-adic Henselizations of finitely presented $A$-algebras $B_0,C_0$ such that $\spec(C_0)$ is an affine open in $X$, $\spec(B_0)$ is an affine open in $Y$. 

If $$B_0 \simeq A[X_1,\dots,X_n]/(f_1,\dots,f_m),\;C_0 \simeq A[Y_1,\dots, Y_r]/(g_1,\dots,g_s),$$ it follows that  $$B \simeq A\{X_1,\dots,X_n\}/(f_1,\dots,f_m),\;C \simeq A\{Y_1,\dots, Y_r\}/(g_1,\dots,g_s).$$ 

As Henselization commutes with filtered colimits \spcite{0A04}{Lemma} we see that $$A\{X_1,\dots,X_n\} = \colim A_i\{X_1,\dots,X_n\}, \;A\{Y_1,\dots,Y_r\} = \colim A_i\{Y_1,\dots,Y_r\}.$$ Therefore we can find $i_0$ and elements $$F_1,\dots,F_m \in A_{i_0}\{X_1,\dots,X_n\},\; G_1,\dots,G_s \in A_{i_0}\{Y_1,\dots,Y_r\}$$ such that the image of the $F_j$ in $A\{X_1,\dots,X_n\}=\colim A_i\{X_1,\dots,X_n\}$ is $f_j$ (with a similar statement for the elements $G_j, g_j$). 

Furthermore, if $h_i := \varphi(X_i) \in C$, we can find $H_1,\dots,H_n \in A_{i_0}\{Y_1,\dots,Y_r\}/(G_1,\dots,G_s)$ such that the image of $H_j$ in $A_{i_0}\{Y_1,\dots,Y_r\}/(G_1,\dots,G_s) \otimes_{A_{i_0}} A = C$ is $h_j$. 

Then we can define a map $\phi_{i_0}: A_{i_0}[X_1,\dots,X_n] \to A_{i_0}\{Y_1,\dots,Y_r\}/(G_1,\dots,G_s)=:C_{i_0}$ by $\phi_{i_0}(X_j)=H_j$. This map factors through $A_{i_0}\{X_1,\dots,X_n\}$ by the universal property of Henselization, giving a map $\varphi'_{i_0}: A_{i_0}\{X_1,\dots,X_n\} \to C_{i_0}$. 

We see that the elements $\varphi'_0(F_j)$ must be $0$ after tensoring with $A$, or after taking the colimit; thus we can increase $i_0$ as needed so that $\varphi'_{i_0}(F_j)=0 \in C_{i_0}$ for all $j$. This gives us a map $\varphi_{i_0}: B_{i_0} := A_{i_0}\{X_1,\dots,X_n\}/(F_1,\dots,F_m) \to C_{i_0}$ which is by definition equal to $\varphi: B \to C$ after tensoring with $A$. 

Hence in the case $X^h,Y^h$ are Henselian-affine, we can use the map $f_{i_0}: \sph(C_{i_0},I_{i_0}C_{i_0}) \to \sph(B_{i_0},I_{i_0}B_{i_0})$ arising from $\varphi_{i_0}$. 

Before we prove the general case, we make the following claim: 

\noindent\textit{Claim:} If $\psi_{i_0}: B_{i_0} \to C_{i_0}$ is another $A_{i_0}$-map of Henselian pairs such that $\psi_{i_0} \otimes_{A_{i_0}} A = \varphi_{i_0} \otimes_{A_{i_0}} A = \varphi$, then there exists some index $i_1 \ge i_0$ such that $\varphi_{i_0} \otimes_{A_{i_0}} A_{i_1} = \psi_{i_0} \otimes_{A_{i_0}} A_{i_1}$. 

\begin{proof}[Proof of Claim] Note that since $B_{i_0} = A_{i_0}\{X_1,\dots,X_n\}/(F_1,\dots,F_m)$, the map $\varphi_{i_0}$ is determined by $\alpha_j := \varphi_{i_0}(X_j) \in C_{i_0}$. Similarly $\psi_{i_0}$ is determined by the elements $\beta_j := \psi_{i_0}(X_j) \in C_{i_0}$. Because $\psi_{i_0} \otimes_{A_{i_0}} A = \varphi_{i_0} \otimes_{A_{i_0}} A = \varphi$, this means that for each $j$ the image of $\alpha_j$ and $\beta_j$ are equal in $C=C_{i_0} \otimes_{A_{i_0}} A$. Therefore we can choose $i_1 \ge i_0$ so that the images of $\alpha_j,\beta_j$ are equal in $C_{i_1}=C_{i_0} \otimes_{A_{i_0}} A_{i_1}$, which would mean that $\varphi_{i_0} \otimes_{A_{i_0}} A_{i_1} = \psi_{i_0} \otimes_{A_{i_0}} A_{i_1}$. \end{proof}

The general case can be deduced from the affine case by gluing as follows. By \spcite{01ZM}{Lemma}, we can find an index $i_0$ and finitely presented $A_{i_0}$-schemes $X_{i_0}, Y_{i_0}$ such that $X=X_{i_0} \otimes_{A_{i_0}} A, Y = Y_{i_0} \otimes_{A_{i_0}} A$. Let $Y_{i_0}=\bigcup_j V_{i_0,j}$ be a finite affine open cover of $Y_{i_0}$, with $Y_{i_0}^h=\bigcup_j V_{i_0,j}^h, Y=\bigcup V_j, Y^h=\bigcup V_j^h$ the corresponding finite (Henselian) affine open covers of $Y_{i_0}^h, Y, Y^h$. 

For each $j$, we use the fact that $X$ and hence $X^h$ are qcqs to cover $f^{-1}(V_{j}^h)$ with finitely many Henselian affine opens $f^{-1}(V_{j}^h)=\bigcup_k U_{j,k}^h$, where each Henselian affine open $U_{j,k}^h$ is the Henselization of an affine open $U_{j,k} \subseteq X$. Possibly increasing $i_0$, we can assume that each affine open $U_{j,k}$ arises from an affine open $U_{i_0,j,k} \subseteq X_{i_0}$. 

Then for each map $f|_{U_{j,k}^h}: U_{j,k}^h \to V_j^h$, we can use the affine case of the lemma which we have proved to obtain maps $f_{i_0,j,k}: U_{i_0,j,k}^h \to V_{i_0,j}^h$ such that $f_{i_0,j,k} \otimes_{A_{i_0}} A = f|_{U_{j,k}^h}$ after possibly increasing $i_0$. The intersection of any two $U_{i_0,j,k}^h$ can be covered with finitely many Henselian affine opens as well, so using the claim above, we can increase $i_0$ so that the $f_{i_0,j,k}$ are equal on each of these intersections. This allows us to glue the maps together to obtain a map $f_{i_0}: X_{i_0}^h \to Y_{i_0}^h$ which accomplishes the desired.
\end{proof}

\begin{lemma}\label{lem:subschemelem} In the setting of Lemma~\ref{lem:schemelem}, if $Z$ is a Henselian scheme which is \hlfp and qcqs over $\sph(A)$ with an \hlfp closed Henselian subscheme $\ol{Z} \subseteq Z$, then there exists an index $i_0$ and a Henselian scheme $Z_{i_0}$, which is \hlfp over $\sph(A_{i_0})$, with an \hlfp closed Henselian subscheme $\ol{Z}_{i_0}$ such that $Z = Z_{i_0} \otimes_{A_{i_0}} A, \ol{Z} = \ol{Z}_{i_0} \otimes_{A_{i_0}} A$. 
\end{lemma}
\begin{proof} Locally, the Henselian scheme $Z$, which is \hlfp over $\sph(A)$, can be described as $\sph(B)$ for an \hfp $A$-algebra $B$, which is the $I$-adic Henselization of a finite presented $A$-algebra $B_0=A[X_1,\dots,X_n]/(f_1,\dots,f_m)$, so $B = A\{X_1,\dots,X_n\}/(f_1,\dots,f_m)$. We can choose a Henselian affine cover of $Z$ so that on each Henselian affine $\sph(B)$, the \hlfp closed Henselian subscheme $\ol{Z}$ can be described as $\sph(C)$ for $C$ which is the quotient of $B$ by a finitely generated ideal: $C \simeq A\{X_1,\dots,X_n\}/(f_1,\dots,f_m, g_1,\dots,g_r)$ (where the $g_j$ lie in $A\{X_1,\dots,X_n\}$).

As Henselization commutes with filtered colimits \spcite{0A04}{Lemma} we see that $$A\{X_1,\dots,X_n\} = \colim A_i\{X_1,\dots,X_n\}.$$ Thus we can find $i_0$ and elements $F_1,\dots,F_m, G_1,\dots, G_r \in A_{i_0}\{X_1,\dots,X_n\}$ such that the image of the $F_j$ in $A\{X_1,\dots,X_n\}=\colim A_i\{X_1,\dots,X_n\}$ is $f_j$ (with a similar statement for the elements $G_j, g_j$). 

It then follows that for $$B_{i_0}=A_{i_0}\{X_1,\dots,X_n\}/(F_1,\dots,F_m), C_{i_0} = A_{i_0}\{X_1,\dots,X_n\}/(F_1,\dots,F_m,G_1,\dots,G_r)$$ we have $B_{i_0} \otimes_{A_{i_0}} A \simeq B, C_{i_0} \otimes_{A_{i_0}} A \simeq C$. Therefore in the case where $Z,\ol{Z}$ are affine, we can use the closed Henselian subscheme $\ol{Z}_{i_0}=\sph(C_{i_0},I_{i_0}C_{i_0})$ inside $Z_{i_0}=\sph(B_{i_0},IB_{i_0})$.  

The general case can be deduced from the affine case by gluing as follows. Since $Z$ is qcqs, we can choose a finite Henselian affine open cover $Z=\bigcup_j V_j$ such that, as above, each of the intersections $\ol{Z} \cap V_j =: \ol{V}_j$ is a closed immersion $\ol{V}_j \to V_j$ defined by an ideal of $\O(V_j)$.

For each pair $j,k$  we can cover the intersection $V_j \cap V_k$ with finitely many Henselian affine opens $V_j \cap V_k = \bigcup_\ell U_{j,k,\ell}$ such that each $U_{j,k,\ell}$ is a distinguished affine open both in $V_j$ and in $V_k$ \hfcite{Proposition 4.2.1}. Since any Henselian affine scheme $\sph(R,J)$ over $\sph(A)$ is the Henselization of an affine scheme $\spec(R)$ over $\spec(A)$ at a closed subscheme $\spec(R/J)$, we see that we can apply the affine case of Lemma~\ref{lem:schemelem} to each of the inclusion maps $f_{j,k,\ell}: U_{j,k,\ell} \to V_j$. We thus obtain an index $i_0$ and Henselian affine schemes $V_{i_0,j}, U_{i_0,j,k,\ell}$ over $\sph(A)$ with maps $f_{i_0,j,k,\ell}: U_{i_0,j,k,\ell} \to V_{i_0,j}$ such that $f_{i_0,j,k,\ell} \otimes_{A_{i_0}} A = f_{j,k,\ell}$. 

Since $U_{j,k,\ell}$ is a distinguished affine open in $V_j$, we see that if $V_j=\sph(R)$ for an \hfp $A$-algebra $R=A\{T_1,\dots,T_n\}/(\alpha_1,\dots,\alpha_m)$, then $U_{j,k,\ell}=\sph(R_\alpha^h)$ for some $\alpha \in R$. Then $$U_{j,k,\ell}=\sph(R_\alpha^h)\simeq\sph(R\{T\}/(1-\alpha T))\simeq\sph(A\{T_1,\dots,T_{n+1}\}/(\alpha_1,\dots,\alpha_m,1-\alpha T_{n+1});$$ following the construction in the proof of Lemma~\ref{lem:schemelem} we see that (after possibly increasing $i_0$) the map $f_{i_0,j,k,\ell}: U_{i_0,j,k,\ell} \to V_{i_0,j}$ will arise from a map of \hfp $A_{i_0}$-algebras $$R_{i_0} = A_{i_0}\{T_1,\dots,T_n\}/(\beta_1,\dots,\beta_m) \to A_{i_0}\{T_1,\dots,T_{n+1}\}/(\beta_1,\dots,\beta_m,1-\beta T_{n+1})$$ for some $\beta \in R_{i_0}$. In particular, $f_{i_0,j,k,\ell}$ is also an open immersion, as $$U_{i_0,j,k,\ell}=A_{i_0}\{T_1,\dots,T_{n+1}\}/(\beta_1,\dots,\beta_m,1-\beta T_{n+1})\simeq R_{i_0}\{T\}/(1-\beta T) \simeq (R_{i_0})_\beta^h$$ is a distinguished affine open in $V_{i_0,j}$. Furthermore, we can use the claim in the proof of Lemma~\ref{lem:schemelem} to ensure that the $f_{i_0,j,k,\ell}$ satisfy an appropriate cocycle condition on triple overlaps by increasing $i_0$. 

Then we can ``glue'' the $V_{i_0,j}$ with the maps $f_{i_0,j,k,\ell}: U_{i_0,j,k,\ell} \to V_{i_0,j}$ to obtain a Henselian scheme $Z_{i_0} = \bigcup_j V_{i_0,j}$, which will be \hlfp over $\sph(A_{i_0})$ by construction, and which will satisfy $Z_{i_0} \otimes_{A_{i_0}} A = Z$. Furthermore, we can increase $i_0$ and apply the affine case of this lemma to obtain for each $j$ a closed immersion of Henselian schemes $\ol{V}_{i_0,j} \subseteq V_{i_0,j}$ with $\ol{V}_{i_0,j} \otimes_{A_{i_0}} A = \ol{V}_j$, which we can ``glue'' in a similar way to get a closed subscheme $\ol{Z}_{i_0}$ such that $\ol{Z} = \ol{Z}_{i_0} \otimes_{A_{i_0}} A$. \end{proof}

\begin{lemma}\label{lem:sheaflem} In the setting of Lemma~\ref{lem:schemelem}, if $\F,\G$ are finitely presented $\O_{X^h}$-modules with a map $g: \F \to \G$, then there exists an index $i_0$ and a finitely presented $A_{i_0}$-scheme $X_{i_0}$ with a map of $\O_{X_{i_0}^h}$-modules $g_{i_0}: \F_{i_0} \to \G_{i_0}$ (where $X_{i_0}^h$ is the $I_{i_0}$-adic Henselization) such that \begin{enumerate}[(i)] \item $X=X_{i_0} \otimes_{A_{i_0}} A, \F=\F_{i_0} \times_{X_{i_0}^h} X^h,   \G=\G_{i_0} \times_{X_{i_0}^h} X^h;$ \item the map $g_{i_0}$ pulls back to $g$. \end{enumerate}

Furthermore, we can ensure that if $g$ is an isomorphism, so is $g_{i_0}$.
\end{lemma}
\begin{proof} 
We first consider the case of a Henselian affine scheme $V=\sph(B)$ for $B$ an \hfp $A$-algebra, such that $\F|_V \simeq \widetilde{M}$ for $M$ a finitely presented $B$-module and $\G|_V \simeq \widetilde{N}$ for another finitely presented $B$-module $N$, with the map of $\O_{X^h}$-modules $g: \F \to \G$ arising from a $B$-module map $f: M \to N$.

As in the proof of Lemma~\ref{lem:schemelem}, we can find an index $i_0$ and an \hfp $A_{i_0}$-algebra $B_{i_0}$ so that $B_{i_0} \otimes_{A_{i_0}} A \simeq B$. Letting $B_i = B_{i_0} \otimes_{A_{i_0}} A_i$ for $i \ge i_0$, we see that $\colim_{i \ge i_0} B_i \simeq B$. 

Then by \spcite{05LI}{Lemma} we can find an index $i_1$ and a map of finitely presented $B_{i_1}$-modules $f_{i_1}: M_{i_1} \to N_{i_1}$ such that $M \simeq M_{i_1} \otimes_{B_{i_1}} B, N \simeq N_{i_1} \otimes_{B_{i_1}} B$ and the map $f: M \to N$ is equal to the map $f_{i_1} \otimes 1: M_{i_1} \otimes_{B_{i_1}} B \to N_{i_1} \otimes_{B_{i_1}} B$. Furthermore, if $f$ is an isomorphism, we can ensure $f_{i_1}$ is as well \spcite{05LI}{Lemma}.

Similarly to the proof of Lemma~\ref{lem:schemelem}, it will be useful to note that if for two maps $f_{i_1}, f_{i_1}': M_{i_1} \to N_{i_1}$ the maps $f_{i_1} \otimes 1, f_{i_1}' \otimes_1:  M_{i_1} \otimes_{B_{i_1}} B \to N_{i_1} \otimes_{B_{i_1}} B$ are equal, then for some index $i_2 \ge i_1$ the maps $f_{i_1} \otimes 1, f_{i_1}' \otimes_1:  M_{i_1} \otimes_{B_{i_1}} B_{i_2} \to N_{i_1} \otimes_{B_{i_1}} B_{i_2}$ are equal. (This is proved by considering the images of some finite set of generators of the finitely presented $B_{i_1}$-module $M_{i_1}$ via the maps $f_{i_1},f_{i_1}'$.)

In the general case, we can cover $X^h$ with finitely many Henselian affine opens $V=\sph(B) \subset X^h$ (for $B$ an \hfp $A$-algebra) such that we have an isomorphism $\F|_V \simeq \widetilde{M}$ for $M$ a finitely presented $B$-module. Possibly shrinking the affine opens $V$, we can assume that we also have $\G|_V \simeq \widetilde{N}$ for another finitely presented $B$-module $N$, with the map of $\O_{X^h}$-modules $g: \F \to \G$ arising from a $B$-module map $f: M \to N$. 

As in Lemma~\ref{lem:schemelem}, we obtain an index $i_0$ and a finitely presented $A_{i_0}$-scheme $X_{i_0}$ such that $X_{i_0} \otimes_{A_{i_0}} A = X$. We can increase $i_0$ so for each of the Henselian affine open $V=\sph(B) \subseteq X^h$, the corresponding affine open in $X_{i_0}$ is $U_{i_0}=\spec(R_{i_0})$ for a finitely presented $A_{i_0}$-algebra $R_{i_0}$ such that $R_{i_0}^h \otimes_{A_{i_0}} A = B$. Possibly replacing $X_{i_0}$ with an open subscheme, we can ensure that $X_{i_0}$ is covered by the affine opens $U_{i_0}=\spec(R_{i_0})$ and $X_{i_0}^h$ is covered by Henselian affine opens $V_{i_0}=\sph(B_{i_0})$ for $B_{i_0}=R_{i_0}^h$.

By the affine case of this lemma, we can increase $i_0$ so that on each of the Henselian affine opens $V$, the map of $B$-modules $f: M \to N$ corresponding to $\F|_V \to \G|_V$ arises from a map of $B_{i_0}$-modules $f_{i_0}: M_{i_0} \to N_{i_0}$; these modules $M_{i_0}, N_{i_0}$ can be ``glued'' together similarly to the ``gluing'' of Henselian affine schemes in Lemma~\ref{lem:subschemelem} (possibly increasing $i_0$ to ensure compatibility on overlaps of the $V_{i_0}$) to get $\O_{X_{i_0}^h}$-modules $\F_{i_0},\G_{i_0}$. In the same way we may ``glue'' the maps $f_{i_0}$ to a map $g_{i_0}: \F_{i_0} \to \G_{i_0}$ (again, possibly increasing $i_0$ to ensure compatibility on overlaps) which will pull back to $g: \F \to \G$ by construction. 

Since for each $f: M \to N$ we can ensure that $f_{i_0}$ is an isomorphism if $f$ is, the same is true for $g_{i_0},g$.\end{proof}

\nocite{pop}

\printbibliography \end{document}